\documentclass[reqno,final]{amsart}
\usepackage{amssymb,latexsym}
\usepackage{url}
\usepackage{xspace}
\usepackage{tensor} 
\usepackage[hidelinks]{hyperref}

\newcommand{\CC}{\ensuremath{{\mathcal{A}}}}
\newcommand{\CCp}[1]{\ensuremath{{\mathcal{A}^{{#1}}}}}
\newcommand{\DC}{\ensuremath{{\mathcal{A}_d}}}
\newcommand{\DCp}[1]{\ensuremath{{\mathcal{A}_d^{#1}}}}
\newcommand{\DCl}[1]{\ensuremath{{\mathcal{A}_{d,{#1}}}}}
\newcommand{\DClp}[2]{\ensuremath{{\mathcal{A}_{d,{#1}}^{{#2}}}}}
\newcommand{\LDCp}[1]{\ensuremath{{a_d^{#1}}}}
\newcommand{\LDClp}[2]{\ensuremath{{a_{d,{#1}}^{#2}}}}
\newcommand{\BC}{\ensuremath{{\mathcal{B}}}}
\newcommand{\BCp}[1]{\ensuremath{{\mathcal{B}^{{#1}}}}}
\newcommand{\BD}{\ensuremath{{\mathcal{B}_d}}}
\newcommand{\BDp}[1]{\ensuremath{{\mathcal{B}_d^{#1}}}}
\newcommand{\BDl}[1]{\ensuremath{{\mathcal{B}_{d,{#1}}}}}
\newcommand{\BDlp}[2]{\ensuremath{{\mathcal{B}_{d,{#1}}^{#2}}}}
\newcommand{\LBDp}[1]{\ensuremath{{b_d^{#1}}}}
\newcommand{\LBDlp}[2]{\ensuremath{{b_{d,{#1}}^{#2}}}}
\newcommand{\Holp}[1]{\ensuremath{\Hol_{{#1}}}}
\newcommand{\HLds}[1]{\ensuremath{\overline{h_d^{{#1}}}}}
\newcommand{\HLd}[1]{\ensuremath{{h_d^{{#1}}}}}
\newcommand{\HLdl}[2]{\ensuremath{{h_{d,{#2}}^{{#1}}}}}
\newcommand{\Hold}{\ensuremath{\Holp{\DC}}}

\DeclareMathOperator{\HPc}{\Phi_\CC}
\DeclareMathOperator{\HPd}{\Phi_{\mathit{d}}}
\DeclareMathOperator{\PTc}{PT_\CC}
\DeclareMathOperator{\PTd}{PT_{\mathit{d}}}
\DeclareMathOperator{\Hol}{Hol}
\DeclareMathOperator{\Log}{Log}
\DeclareMathOperator{\Arg}{Arg}
\DeclareMathOperator{\jmod}{mod}

\newcommand{\SG}{\ensuremath{G}\xspace}
\newcommand{\jgg}{\ensuremath{\mathfrak{g}}\xspace}
\newcommand{\UDt}{\ensuremath{\mathcal{U}}\xspace}
\newcommand{\VDt}{\ensuremath{\mathcal{V}}\xspace}
\newcommand{\WDt}{\ensuremath{\mathcal{W}}\xspace}

\newcommand{\baseRing}[1]{\ensuremath{\mathbb{#1}}}
\newcommand{\N}{\baseRing{N}}

\newcommand{\R}{\baseRing{R}}
\newcommand{\C}{\baseRing{C}}
\newcommand{\NZ}{\ensuremath{\N\cup\{0\}}}

\newcommand{\VD}{\ensuremath{\mathcal{V}_d}\xspace}
\newcommand{\ie}{\textsl{i.e.}\xspace}
\newcommand{\del}{\ensuremath{\partial}}
\newcommand{\ti}[1]{\widetilde{#1}}
\newcommand{\stext}[1]{\ensuremath{\quad\text{#1}\quad}}
\newcommand{\jdef}[1]{\index{#1}\emph{#1}}
\newcommand{\conj}{\overline}
\newcommand{\lie}[1]{\ensuremath{Lie(#1)}}
\providecommand{\abs}[1]{\ensuremath{\left\lvert{#1}\right\rvert}}
\newcommand{\SM}{\ensuremath{\smallsetminus}}

\theoremstyle{plain}
\newtheorem{theorem}{Theorem}[section]
\newtheorem{corollary}[theorem]{Corollary}
\newtheorem{prop}[theorem]{Proposition}
\newtheorem{lemma}[theorem]{Lemma}

\theoremstyle{definition}
\newtheorem{definition}[theorem]{Definition}
\newtheorem{remark}[theorem]{Remark}
\newtheorem{example}[theorem]{Example}

\numberwithin{equation}{section}

\newcommand{\FP}[2]{\tensor[_{{#1}}]{\times}{_{{#2}}}} 
\newcommand{\PBmap}{\ensuremath{\pi}\xspace}
\newcommand{\PBtotal}{\ensuremath{Q}\xspace}
\newcommand{\PBbase}{\ensuremath{M}\xspace}

\begin{document}


\title[Discrete connections: abelian case]{Discrete connections on
  principal bundles:\\ abelian group case}

\author{Javier Fern\'andez}
\address{Instituto Balseiro\\Universidad Nacional de Cuyo --
  C.N.E.A.\\Bariloche\\R{\'\i}o Negro\\Rep\'ublica Argentina}
\email{jfernand@ib.edu.ar}

\author{Mariana Juchani} \address{Departamento de Matem\'atica,
  Facultad de Ciencias Exactas, Universidad Nacional de La Plata\\50 y
  115, La Plata, Buenos Aires, 1900, Rep\'ublica Argentina\\Centro de
  Matem\'atica de La Plata (CMaLP)}
\email{marianaevaj@gmail.com}

\author{Marcela Zuccalli}
\address{Departamento de Matem\'atica,
  Facultad de Ciencias Exactas, Universidad Nacional de La Plata\\50 y
  115, La Plata, Buenos Aires, 1900, Rep\'ublica Argentina\\Centro de
  Matem\'atica de La Plata (CMaLP)}
\email{marce@mate.unlp.edu.ar}

\thanks{This research was partially supported by grants from the
  Universidad Nacional de Cuyo (grant 06/C567), Universidad Nacional
  de La Plata, and CONICET}




\begin{abstract}
  In this note we consider a few interesting properties of discrete
  connections on principal bundles when the structure group of the
  bundle is an abelian Lie group. In particular, we show that the
  discrete connection form and its curvature can be interpreted as
  singular $1$ and $2$ cochains respectively, with the curvature being
  the coboundary of the connection form. Using this formalism we prove
  a discrete analogue of a formula for the holonomy around a loop
  given by Marsden, Montgomery and Ratiu for (continuous) connections
  in a similar setting.
\end{abstract}

\bibliographystyle{amsalpha}

\maketitle


\section{Introduction}
\label{sec:introduction}

Principal bundles are used to state many important questions in
Geometry and in Physics. One versatile tool in the study of those
bundles, and in finding answers to the questions, are the principal
connections and the ``connection package'': curvature, parallel
transport and holonomy. Curvature and holonomy are closely related
notions, as shown, for example, by the Ambrose--Singer Theorem
(Thm. 8.1 in~\cite{bo:kobayashi_nomizu-foundations-v1}). When the
structure group $G$ of the (left) principal bundle
$\PBmap:\PBtotal\rightarrow \PBbase$ is abelian and $\CC$ is a
principal connection on $\PBmap$, a well known formula (see, for
instance, p. 41
of~\cite{ar:marsden_montgomery_ratiu-reduction_symmetry_and_phases_in_mechanics})
gives a direct expression for the holonomy of $\CC$ around a loop
$\rho$ in $\PBbase$ in terms of an integral of $\CC$ over $\rho$ or,
in case $\rho$ is the boundary of a surface $\sigma$, the integral
over $\sigma$ of the curvature of $\CC$. This kind of formula is
very useful for many applications: for instance, it allows the
control of the ``displacement'' in a fiber by choosing an adequate
loop $\rho$ and parallel-transporting along $\rho$. It is also useful
in the reconstruction of the dynamics of some systems on $\PBtotal$
that have $G$ as a symmetry group and whose (reduced) dynamics is
known in $\PBbase:=\PBtotal/G$ (see \S5 of
~\cite{ar:marsden_montgomery_ratiu-reduction_symmetry_and_phases_in_mechanics}).
In this setting the holonomy is also known as the \jdef{geometric
  phase} (see, for
instance,~\cite{ar:simon-holonomy_the_quantum_adiabatic_theorem_and_berrys_phase}).

More precisely, if $V\subset \PBbase$ is an open subset and
$s:V\rightarrow \PBtotal$ is a smooth local section of $\PBmap$, let
$\rho:[0,1]\rightarrow \PBbase$ be a continuous loop contained in $V$
and $\conj{m}:=\rho(0)$; pick $\conj{q}\in \PBtotal|_{\conj{m}}$, the
fiber of $\PBtotal$ over $\conj{m}$, and define
$\Phi_\CC(\rho,\conj{q}) \in G$ by
\begin{equation}\label{eq:phase_in_terms_of_connection-cont}
  \HPc(\rho,\conj{q}) := \exp_G\left(-\int_{\rho} \CCp{s}\right),
\end{equation}
where $\CCp{s}:=s^*(\CC)$ is the local expression of $\CC$ in the
trivialization induced by $s$ (hence a $1$-form on $V$ with values in
$\jgg:=\lie{G}$) and $\exp_G:\jgg\rightarrow G$ is the Lie-theoretic
exponential map. If, in addition, $\rho$ is the boundary of a surface
$\sigma$ contained in $V$, by Stoke's Theorem, we have
\begin{equation}\label{eq:phase_in_terms_of_curvature-cont}
  \HPc(\rho,\conj{q}) = \exp_G\left(-\int_{\sigma} \BCp{s}\right),
\end{equation}
for $\BCp{s} := s^*(\BC) = d \CCp{s}$, where $\BC$ is the curvature
$2$-form of $\CC$. These formulas are relevant because
$\HPc(\rho,\conj{q})$ is the phase gained by $\conj{q}$ on parallel
transport around the loop $\rho$. In other words, if
$\PTc(\rho):\PBtotal|_{\conj{m}}\rightarrow \PBtotal|_{\conj{m}}$ is
the parallel transport over $\rho$ operator (with respect to $\CC$),
then
\begin{equation*}
  \PTc(\rho)(\conj{q}) = l^\PBtotal_{\HPc(\rho,\conj{q})}(\conj{q}),
\end{equation*}
where $l^\PBtotal_g$ is the left $G$-action on $\PBtotal$ defined by
the principal $G$-bundle structure.

As we mentioned above, connections are useful in the study of certain
symmetric dynamical systems, for example, the mechanical systems as
considered in~\cite{bo:AM-mechanics}. In many instances, it is
essential to consider numerical approximations to those systems, which
can be thought of as discrete-time dynamical systems
(see~\cite{ar:marsden_west-discrete_mechanics_and_variational_integrators}). In
this context it is natural to replace the tangent bundle $T\PBtotal$
with a ``discrete version'' $\PBtotal\times \PBtotal$\footnote{The
  basic motivation is that if time is discrete, a velocity vector
  ---that is, an element of $T\PBtotal$--- may be replaced by two
  nearby points ---that is, an element of $\PBtotal\times \PBtotal$,
  near the diagonal $\Delta_\PBtotal$.}. If, in addition, there is a
symmetry group $G$ acting on $\PBtotal$ in such a way that the
quotient map $\PBmap:\PBtotal\rightarrow \PBtotal/G$ is a principal
$G$-bundle, in the continuous case, a principal connection $\CC$ on
$\PBmap$ can be used to decompose all elements of $T\PBtotal$ into
vertical an horizontal parts. Discrete connections were introduced by
M. Leok, J. Marsden and A. Weinstein
in~\cite{ar:leok_marsden_weinstein-a_discrete_theory_of_connections_on_principal_bundles}
and~\cite{th:leok-thesis} and, later, refined by some of us
in~\cite{ar:fernandez_zuccalli-a_geometric_approach_to_discrete_connections_on_principal_bundles}
in order to have a geometric way of splitting elements of
$\PBtotal\times \PBtotal$ into ``vertical'' and ``horizontal'' parts
in a way that was suitable to perform the symmetry reduction for the
discrete time mechanical systems. With this idea in mind, the goal of
this paper is to prove formulas analogous
to~\eqref{eq:phase_in_terms_of_connection-cont}
and~\eqref{eq:phase_in_terms_of_curvature-cont} for discrete
connections on a principal $G$-bundle, when $G$ is abelian. As for
(continuous) connections, a natural application of such formulas is to
control the dynamics of symmetric dynamical systems.

Even though there are many conceptual parallels between connections
and discrete connections, the fact that (for abelian $G$) the
curvature ($2$-form) $\BC$ is related to the connection $1$-form $\CC$
by $\BC=d\CC$ has no equivalent in the discrete world: both the
discrete connection form $\DC$ and its curvature $\BD$ are $G$-valued
functions. In order to reach our stated goal we introduce a formalism
that allows us to view $\DC$ and $\BD$ as singular cochains $[\DC]$
and $[\BD]$ such that $[\BD]=\delta [\DC]$. It is within this
framework that we obtain the discrete holonomy phase
formulas~\eqref{eq:phases_and_loops-group-DC}
and~\eqref{eq:phases_and_loops-group-BD}. We mention that, in this
setting, what we call integration is the natural pairing of singular
cochains and chains; this idea is in line with the fact that the
integral of differential forms over manifolds is a realization of the
pairing between cochains of differential forms ---the elements of the
de Rham complex--- and singular chains on a manifold. There is still a
twist in that~\eqref{eq:phase_in_terms_of_connection-cont}
and~\eqref{eq:phase_in_terms_of_curvature-cont} involve integration in
$\jgg$ while~\eqref{eq:phases_and_loops-group-DC}
and~\eqref{eq:phases_and_loops-group-BD} use integration of $G$-valued
cochains. Thus, we reformulate our $G$-valued singular cochains as
$\jgg$-valued ones and are able to
obtain~\eqref{eq:phases_and_loops-algebra-DC}
and~\eqref{eq:phases_and_loops-algebra-BD}, the exact analogue of the
continuous phase formulas.

The plan for the paper is as follows. In
Section~\ref{sec:discrete_connections_on_principal_bundles} we recall
the relevant notions of discrete connection on a principal $G$-bundle,
its curvature, the corresponding parallel transport and holonomy. In
Section~\ref{sec:singular_cohomology_and_adaptations} we review some
basic ideas of singular chains and cochains which we then refine to
what we call the \jdef{small complexes} that are needed in order to
view the discrete connection form $\DC$ and the discrete curvature
$\BD$ as singular $1$ and $2$ cochains $[\DC]$ and $[\BD]$
respectively. In Section~\ref{sec:holonomy_around_a_loop} we obtain
formulas to compute the discrete holonomy phase that, eventually, lead
to the first version of our result expressing those phases in terms of
integrals of the $G$-valued $[\DC]$ and $[\BD]$
(Theorem~\ref{th:phases_and_loops-group}). Last, in
Section~\ref{sec:forms_with_values_in_the_lie_algebra} we consider
$\jgg$-valued singular cochains and construct logarithmic versions of
$[\DC]$ and $[\BD]$, which appear in the expression of the discrete
holonomy phase as an integral of $\jgg$-valued cochains
(Theorem~\ref{th:phases_and_loops-algebra}).

\emph{Notation:} throughout this paper
$\PBmap:\PBtotal\rightarrow \PBbase$ is a smooth principal (left)
$G$-bundle ---usually referred to as $\pi$ in what follows--- and we
denote the (left) $G$-action on $\PBtotal$ by $l^\PBtotal_g(q)$ for
$q\in \PBtotal$ and $g\in G$. In addition to $l^\PBtotal$ we will
consider several other (left) $G$-actions: the lifted action on
$T\PBtotal$, the diagonal action on $\PBtotal\times \PBtotal$ and the
action on the second variable of $\PBtotal\times
\PBtotal$. Explicitly, for $g\in G$, $v_q\in T_q\PBtotal$ and
$(q_0,q_1)\in \PBtotal\times \PBtotal$,
\begin{gather*}
  l^{T\PBtotal}_g(v_q) := T_ql^\PBtotal_g(v_q),\quad l^{\PBtotal\times
    \PBtotal}_g(q_0,q_1):=(l^\PBtotal_g(q_0),l^\PBtotal_g(q_1)),\\
  l^{\PBtotal\times \PBbase}_g(q,m) :=(l^\PBtotal_g(q),m), \stext{ and
  } l^{\PBtotal\times \PBtotal_2}_g(q_0,q_1):=(q_0,l^\PBtotal_g(q_1)).
\end{gather*}
We denote the \jdef{diagonal} of any Cartesian product $X\times X$ by
$\Delta_X$ and the \jdef{discrete vertical submanifold} of $\PBtotal$
by
$\VD:=(\PBmap\times \PBmap)^{-1}(\Delta_\PBbase)\subset \PBtotal\times
\PBtotal$.


\section{Discrete connections on principal bundles}
\label{sec:discrete_connections_on_principal_bundles}

In this section we review the notion of discrete connection on a
principal bundle (via discrete connection form and discrete horizontal
lift), the associated curvature and the corresponding parallel
transport.


\subsection{Discrete connection form and discrete horizontal lift}
\label{sec:discrete_connection_form_and_discrete_horizontal_lift}

Discrete connections on a principal bundle can be constructed using
different data. For this paper it will be sufficient to characterize
discrete connections via their \jdef{discrete connection form} and
their \jdef{discrete horizontal lift}. For more information on
discrete connections
see~\cite{ar:fernandez_zuccalli-a_geometric_approach_to_discrete_connections_on_principal_bundles}.

\begin{definition}\label{def:D_type_subset}
  An open subset $\UDt\subset \PBtotal\times \PBtotal$ is said to be
  of \jdef{$D$-type} if it is $G\times G$-invariant for the product of
  the $G$-action with itself and $\VD\subset \UDt$ (in particular,
  $\Delta_\PBtotal\subset \UDt$).
\end{definition}

Given a $D$-type subset $\UDt\subset \PBtotal\times \PBtotal$ we define
\begin{equation}
  \label{eq:prime_and_2prime_subsets-def}
  \UDt':=(id_\PBtotal\times \PBmap)(\UDt)\subset \PBtotal\times \PBbase \stext{ and }
  \UDt'':=(\PBmap\times \PBmap)(\UDt)\subset \PBbase\times \PBbase.
\end{equation}
As $\PBmap$ is a principal bundle map, both $\UDt'$ and $\UDt''$ are
open subsets.

\begin{definition}\label{def:discrete_connection_form}
  Let $\UDt\subset \PBtotal\times \PBtotal$ be a $D$-type subset. A
  smooth function $\DC:\UDt\rightarrow \SG$ is called a \jdef{discrete
    connection form} on $\PBmap$ if, for all $q\in \PBtotal$,
  $\DC(q,q)=e$, the identity element of $G$, and it satisfies
  \begin{equation*}
    \DC(l^\PBtotal_{g_0}(q_0),l^\PBtotal_{g_1}(q_1)) = g_1 \DC(q_0,q_1) g_0^{-1}
    \stext{ for all } (q_0,q_1)\in \UDt,\quad g_0,g_1\in G.
  \end{equation*}
\end{definition}

\begin{remark}
  An important difference between continuous and discrete connections
  on a principal $G$-bundle $\PBmap: \PBtotal\rightarrow \PBbase$ is
  that while the (continuous) connection $1$-form is defined over all
  of $T\PBtotal$, a discrete connection over $\PBmap$ is only defined
  in some open neighborhood of the diagonal
  $\Delta_\PBtotal\subset \PBtotal\times \PBtotal$, unless $\PBmap$ is
  a trivial bundle. Indeed, if a discrete connection is globally
  defined, a global section of $\PBmap$ can be readily
  constructed. For this reason, when dealing with discrete
  connections, their domain is very important.
\end{remark}

\begin{remark}
  Just as connections on principal bundles can be defined by a
  horizontal distribution, discrete principal connections can be
  defined by a \jdef{horizontal submanifold}
  $Hor_\DC\subset \PBtotal\times \PBtotal$. Indeed, this is the
  approach
  of~\cite{ar:leok_marsden_weinstein-a_discrete_theory_of_connections_on_principal_bundles}
  and~\cite{ar:fernandez_zuccalli-a_geometric_approach_to_discrete_connections_on_principal_bundles}. But,
  as shown in Theorem 3.6
  of~\cite{ar:fernandez_zuccalli-a_geometric_approach_to_discrete_connections_on_principal_bundles},
  giving a horizontal submanifold is equivalent to giving a discrete
  connection form $\DC$. The horizontal submanifold associated to
  $\DC$ is $Hor_\DC:=\DC^{-1}(\{e\})$.
\end{remark}

\begin{example}\label{ex:DC_mu-DC_form}
  Let $\R_{>0} := (0,+\infty)\subset\R$ and
  $U(1):=\{z\in\C:\abs{z}=1\}$ seen as a multiplicative group under
  complex multiplication. Then, $\PBmap:\PBtotal\rightarrow \PBbase$
  given by $p_1:\R_{>0}\times U(1)\rightarrow \R_{>0}$ is a smooth
  (trivial) principal $G$-bundle, with $G:= U(1)$ acting on the second
  factor by multiplication. Although it is not relevant for this work,
  the manifold $\PBtotal$ can be seen, for instance, as the
  configuration manifold (for the center of mass description) of a
  planar mechanical system consisting of two equal-mass particles. For
  $\mu\in\N$, we define
  \begin{equation*}
    \DCl{\mu}:\PBtotal\times \PBtotal\rightarrow U(1) \stext{ by }
    \DCl{\mu}((r_0,h_0),(r_1,h_1)) := \exp(i(r_1-r_0)^\mu) \frac{h_1}{h_0}.
  \end{equation*}
  It is easy to check that $\DCl{\mu}$ is a discrete connection form
  on $\PBmap$ that is globally defined, \ie, with domain
  $\UDt:=\PBtotal\times \PBtotal$.
\end{example}

\begin{definition}\label{def:HLd}
  Let $\UDt\subset \PBtotal\times \PBtotal$ be of $D$-type. A smooth
  function $\HLd{}:\UDt'\rightarrow \PBtotal\times \PBtotal$ is a
  \jdef{discrete horizontal lift} on $\PBmap$ if the following
  conditions hold.
  \begin{enumerate}
  \item \label{it:HL_properties-HL_is_smooth}
    $\HLd{}:\UDt'\rightarrow \PBtotal\times \PBtotal$ is $G$-equivariant for the
    $G$-actions $l^{\PBtotal\times \PBbase}$ and $l^{\PBtotal\times \PBtotal}$.
  \item \label{it:HL_properties-HL_is_a_section} $\HLd{}$ is a section
    of
    $(id_\PBtotal\times \PBmap):\PBtotal\times \PBtotal\rightarrow
    \PBtotal\times \PBbase$ over $\UDt'$, that is,
    $(id_\PBtotal\times \PBmap) \circ \HLd{} = id_{\UDt'}$.
  \item \label{it:HL_properties-HL_normalization} For every $q\in \PBtotal$,
    $\HLd{}(q,\PBmap(q)) = (q,q)$.
  \end{enumerate}
\end{definition}

If $\HLd{}:\UDt'\rightarrow \PBtotal\times \PBtotal$ is a discrete
horizontal lift on $\PBmap$, we define $\HLds{}:=p_2\circ \HLd{}$,
where $p_2:\PBtotal\times \PBtotal\rightarrow \PBtotal$ is the
projection onto the second factor.

Discrete connection forms and discrete horizontal lifts are related as
follows. First, recall the following construction: consider the fiber
product $\PBtotal\FP{\PBmap}{\PBmap}\PBtotal$ of $\PBmap$ with itself
---that is, the set of pairs $(q_0,q_1)$ such that
$\PBmap(q_0) = \PBmap(q_1)$. Let
$\kappa:\PBtotal\FP{\PBmap}{\PBmap}\PBtotal\rightarrow \SG$ be defined
by $\kappa(q_0,q_1):=g$ if and only if $l^\PBtotal_g(q_0)=q_1$. It is
easy to check that $\kappa$ is a smooth function. Let
$\UDt\subset \PBtotal\times \PBtotal$ be a $D$-type subset. Given a
discrete connection form $\DC:\UDt\rightarrow G$, we define
\begin{equation}\label{eq:HLd_from_DC}
  h_\DC:\UDt'\rightarrow \PBtotal\times \PBtotal \stext{ by }
  h_\DC(q,r) := (q,l^\PBtotal_{\DC(q,q')^{-1}}(q')),
\end{equation}
for any $q'\in \PBtotal|_r$. Conversely, given a discrete horizontal
lift $\HLd{}:\UDt'\rightarrow \PBtotal\times \PBtotal$, we define
\begin{equation}\label{eq:DC_from_HLd}
  \DCp{\HLd{}}:\UDt\rightarrow G \stext{ by }
  \DCp{\HLd{}}(q_0,q_1):=\kappa(\HLds{}(q_0,\PBmap(q_1)),q_1).
\end{equation}

\begin{theorem}\label{th:DC_forms_and_HLd_are_equivalent}
  The maps $h_\DC$ and $\DCp{\HLd{}}$ defined
  by~\eqref{eq:HLd_from_DC} and~\eqref{eq:DC_from_HLd} are a discrete
  horizontal lift and a discrete connection form on $\PBmap$
  respectively. Furthermore, the two constructions are inverses of
  each other.
\end{theorem}

\begin{proof}
  It follows from Theorems 3.6 and 4.6
  of~\cite{ar:fernandez_zuccalli-a_geometric_approach_to_discrete_connections_on_principal_bundles}.
\end{proof}

As a consequence of Theorem~\ref{th:DC_forms_and_HLd_are_equivalent}
discrete connection forms and discrete horizontal lifts are
alternative descriptions of a single object, the \jdef{discrete
  connections} with domain $\UDt$. In this spirit we speak of a
``discrete connection'' with domain $\UDt$ as the object defined by
either one of these maps.

\begin{remark}\label{rem:HL_is_horizontal}
  If $\DC:\UDt\rightarrow G$ is a discrete connection form on $\PBmap$
  and $h_\DC$ is the associated discrete horizontal lift, then
  $\DC(h_{\DC}(q,r)) = e$ for all $(q,r)\in \UDt'$.
\end{remark}

\begin{example}\label{ex:DC_mu-HL}
  The discrete horizontal lift associated by
  Theorem~\ref{th:DC_forms_and_HLd_are_equivalent} to the discrete
  connection form $\DCl{\mu}$ introduced in
  Example~\ref{ex:DC_mu-DC_form} is, using~\eqref{eq:HLd_from_DC},
  \begin{gather*}
    \HLdl{}{\mu}:(\R_{>0}\times U(1))\times \R_{>0}\rightarrow
    (\R_{>0}\times U(1))^2
    \stext{ so that }\\
    \HLdl{}{\mu}((r_0,h_0),r_1) :=
    ((r_0,h_0),(r_1,h_0\exp(-i(r_1-r_0)^\mu))).
  \end{gather*}
\end{example}

Next we introduce a local description of discrete connection forms.

\begin{definition}\label{def:DC_local_expression}
  Let $V\subset \PBbase$ be an open set and $s:V\rightarrow \PBtotal$
  a smooth section of $\PBtotal|_V$. Given a discrete connection
  $\DC:\UDt\rightarrow G$ we define its \jdef{local expression} with
  respect to $s$ by
  \begin{equation}\label{eq:DC-local_expression-def}
    \DCp{s}:\VDt''\rightarrow G \stext{ so that }
    \DCp{s}(m_0,m_1):=\DC(s(m_0),s(m_1)),
  \end{equation}
  where $\VDt'':=(V\times V)\cap \UDt''$
  (recall~\eqref{eq:prime_and_2prime_subsets-def} for $\UDt''$).
\end{definition}

\begin{example}\label{ex:DC_mu-DC_form-local}
  In the context of Example~\ref{ex:DC_mu-DC_form}, the principal
  bundle is trivial. So we can take $V:=\R_{>0}$ and a global section
  of $p_1$, $s(r):=(r,1)$. Then, the ``local'' expression of
  $\DCl{\mu}$ is $\DClp{\mu}{s}(r_0,r_1):=\exp(i(r_1-r_0)^\mu)$ for
  all $(r_0,r_1) \in \VDt'' = \PBbase\times \PBbase = (\R_{>0})^2$.
\end{example}

The following properties of the local expression of a discrete
connection are easy to check.

\begin{lemma}\label{le:properties_of_DC_s}
  Let $\DC:\UDt\rightarrow G$ be a discrete connection on $\PBmap$ and
  $s:V\rightarrow \PBtotal$ be a section of $\PBtotal|_V$. Then,
  $\VDt''\subset \PBbase\times \PBbase$ is open, $\DCp{s}$ is smooth and
  $\DCp{s}(m,m)=e$ for all $m\in V$. If
  $\varphi_s:V\times G\rightarrow \PBtotal$ is defined by
  $\varphi_s(m,g):=l^\PBtotal_g(s(m))$, then
  $\DC(\varphi_s(m_0,g_0),\varphi_s(m_1,g_1)) = g_1 \DCp{s}(m_0,m_1)
  g_0^{-1}$. 
\end{lemma}


\subsection{Curvature of a discrete connection}
\label{sec:curvature_of_a_discrete_connection}

The curvature of a connection on a principal bundle can be seen as an
obstruction to the local trivializability of the bundle (with
connection) or, alternatively, as the obstruction to a certain map
appearing in the Atiyah sequence being a morphism of Lie
algebroids. In a similar manner, a notion of curvature of a discrete
connection on a principal bundle can be introduced as the obstruction
to the bundle (with discrete connection) being locally trivializable
(see~\cite{ar:fernandez_zuccali-holonomy_of_discrete_connections}) or,
alternatively, to a certain map appearing in the discrete Atiyah
sequence being a morphism of local Lie groupoids
(see~\cite{ar:fernandez_juchani_zuccalli-discrete_connections_on_principal_bundles_the_discrete_atiyah_sequence}).

\begin{definition}\label{def:BD}
  Let $\DC:\UDt\rightarrow G$ be a discrete connection form on
  $\PBmap$. The \jdef{curvature} of $\DC$ is the map
  \begin{equation*}
    \BD:\UDt^{(3)}\rightarrow G \stext{ so that }
    \BD(q_0,q_1,q_2) := \DC(q_0,q_2)^{-1} \DC(q_1,q_2) \DC(q_0,q_1),
  \end{equation*}
  where
  $\UDt^{(3)}:=\{(q_0,q_1,q_2)\in \PBtotal^3: (q_j,q_k)\in \UDt \text{
    for all } 0\leq j<k\leq 2\}$. A discrete connection form is
  \jdef{flat} if its curvature is constantly $e$.
\end{definition}

When $V\subset \PBbase$ is an open set and $s:V\rightarrow \PBtotal$
is a smooth section of $\PBtotal|_V$ we define the \jdef{local
  expression of the curvature} with respect to $s$ by
\begin{equation}\label{eq:BD-local_expression-def}
  \BDp{s}:{\VDt''}^{(3)}\rightarrow G \stext{ so that }
  \BDp{s}(m_0,m_1,m_2) := \BD(s(m_0),s(m_1),s(m_2)),
\end{equation}
where
\begin{equation}\label{eq:VDt''(3)-def}
  {\VDt''}^{(3)}:=\{(m_0,m_1,m_2)\in \PBbase^3: (m_j,m_k)\in \VDt'' \text{
    for all } j,k=0,1,2\}.
\end{equation}

\begin{example}\label{ex:DC_mu-curvature}
  The curvature of the discrete connection $\DCl{\mu}$ introduced in
  Example~\ref{ex:DC_mu-DC_form} is
  \begin{equation*}
    \BDl{\mu}((r_0,h_0),(r_1,h_1),(r_2,h_2)) =
    \exp(i(-(r_2-r_0)^\mu + (r_2-r_1)^\mu + (r_1-r_0)^\mu)) ,
  \end{equation*}
  for all $((r_0,h_0),(r_1,h_1),(r_2,h_2)) \in \UDt^{(3)} =
  \PBtotal^3$. Also, in the local trivialization considered in
  Example~\ref{ex:DC_mu-DC_form-local},
  \begin{equation}\label{eq:DC_mu-curvature-local_expression}
    \BDlp{\mu}{s}(r_0,r_1,r_2) =
    \exp(i(-(r_2-r_0)^\mu + (r_2-r_1)^\mu + (r_1-r_0)^\mu)),
  \end{equation}
  for all $(r_0,r_1,r_2) \in {\VDt''}^{(3)} = (R_{>0})^3$. It is easy
  to check that $\BDl{\mu} = 1$ if and only if $\mu=1$.
\end{example}

The following properties of the local expression of the curvature of a
discrete connection are easy to check.

\begin{lemma}\label{le:properties_of_BD_s}
  Let $\DC:\UDt\rightarrow G$ be a discrete connection on $\PBmap$ and
  $s:V\rightarrow \PBtotal$ a section of $\PBtotal|_V$. Then,
  $\BDp{s}$ is smooth and
  \begin{equation*}
    \BDp{s}(m_0,m_1,m_2) = \DCp{s}(m_0,m_2)^{-1} \DCp{s}(m_1,m_2)
    \DCp{s}(m_0,m_1)
  \end{equation*}
  for all $(m_0,m_1,m_2)\in {\VDt''}^{(3)}$. Also,
  \begin{equation*}
    \BD(\varphi_s(m_0,g_0),\varphi_s(m_1,g_1),\varphi_s(m_2,g_2)) = g_0
    \BDp{s}(m_0,m_1,m_2) g_0^{-1}
  \end{equation*}
  for all $(m_0,m_1,m_2)\in {\VDt''}^{(3)}$ and $g_0,g_1,g_2\in G$,
  where $\varphi_s$ is defined in Lemma~\ref{le:properties_of_DC_s}.
\end{lemma}


\subsection{Parallel transport associated to a discrete connection}
\label{sec:parallel_transport_associated_to_a_discrete_connection}

Given a set $X$ and $N\in\NZ$, a \jdef{discrete path} of length $N$ is
an element $x_\cdot :=(x_0,\ldots,x_N)\in X^{N+1}$. The initial and
final points of $x_\cdot$ are $x_0$ and $x_N$ respectively. When
$x_0=x_N$, $x_\cdot$ is a discrete $N$-\jdef{loop}. The set of all
discrete paths of length $N$ with initial point $\conj{x}$ and final
point $\conj{x}'$ is denoted by $\Omega_N(\conj{x},\conj{x}')$ and the
discrete $N$-loops with initial point $\conj{x}$ is denoted by
$\Omega_N(\conj{x})$. Last, we define
$\Omega(\conj{x},\conj{x}') := \cup_{N=0}^\infty
\Omega_N(\conj{x},\conj{x}')$ and
$\Omega(\conj{x}) := \Omega(\conj{x},\conj{x})$.

In what follows, it will be necessary to consider discrete paths that
satisfy some restrictions. Thus, given $U\subset X\times X$ and
$N\in\NZ$ we define the following sets of discrete paths
\jdef{subordinated} to $U$:
$\Omega_{N,U}(\conj{x},\conj{x}'):=\{x_\cdot \in
\Omega_{N}(\conj{x},\conj{x}') : (x_{k-1},x_k)\in U \text{ for all }
k=1,\ldots, N\}$,
$\Omega_{N,U}(\conj{x}) := \Omega_{N,U}(\conj{x},\conj{x})$,
$\Omega_U(\conj{x},\conj{x}') := \cup_{N=0}^\infty
\Omega_{N,U}(\conj{x},\conj{x}')$ and
$\Omega_U(\conj{x}) := \Omega_{U}(\conj{x},\conj{x})$.

Let $\UDt\subset \PBtotal\times \PBtotal$ be a $D$-type subset and
$\HLd{}:\UDt'\rightarrow \PBtotal\times \PBtotal$ be a discrete
horizontal lift on the principal bundle $\PBmap$. Then, for any
discrete path in $\PBbase$,
$m_\cdot \in \Omega_{N,\UDt''}(\conj{m},\conj{m}')$ and
$\conj{q}\in \PBtotal|_{\conj{m}}$ we define inductively
$q_0:=\conj{q}$ and, for each $k=1,\ldots,N$,
$q_k:=\HLds{}(q_{k-1},m_k)$. It is easy to verify that, as
$(m_{k-1},m_k)\in \UDt''$ for all $k=1,\ldots,N$, all
$(q_{k-1},m_k)\in \UDt'$, so that each $q_k$ is well defined and
$\PBmap(q_k)=m_k$. Let
$q_\cdot:=(q_0,\ldots,q_N) \in \Omega_{N}(\conj{q},q_N)$; by
construction, $\PBmap(q_N) = m_N=\conj{m}'$. The discrete path
$q_\cdot$ is called the \jdef{discrete horizontal lift} of $m_\cdot$
starting at $\conj{q}$.  This leads to the following notion.

\begin{definition}\label{def:PTd}
  Let $\UDt\subset \PBtotal\times \PBtotal$ be a $D$-type subset and
  $\HLd{}:\UDt'\rightarrow \PBtotal\times \PBtotal$ be a discrete
  horizontal lift on $\PBmap$. For each
  $m_\cdot \in \Omega_{N,\UDt''}(\conj{m},\conj{m}')$ we define the
  \jdef{discrete parallel transport map} over $m_\cdot$
  \begin{equation*}
    \PTd:\PBtotal|_{\conj{m}}\rightarrow \PBtotal|_{\conj{m}'}
    \stext{ so that }
    \PTd(m_\cdot)(\conj{q}) := q_N,
  \end{equation*}
  where $q_N$ is the one constructed in the previous paragraph.
\end{definition}

Of interest for our analysis is the special case where one considers
parallel transport over discrete loops, so that $\PTd$ is a map from a
fiber of $\PBtotal$ onto itself.

\begin{definition}\label{def:holonomy_around_a_loop}
  Let $\UDt\subset \PBtotal\times \PBtotal$ be a $D$-type subset and
  $\HLd{}:\UDt'\rightarrow \PBtotal\times \PBtotal$ be a discrete
  horizontal lift on $\PBmap$. For each
  $m_\cdot \in \Omega_{N,\UDt''}(\conj{m})$ and
  $\conj{q}\in \PBtotal|_{\conj{m}}$ we define the \jdef{discrete
    holonomy phase} around $m_\cdot$ starting at $\conj{q}$ as
  \begin{equation*}
    \HPd(m_\cdot,\conj{q}) :=
    \kappa(\conj{q},\PTd(m_\cdot)(\conj{q})) \in G.
  \end{equation*}
\end{definition}

The discrete holonomy phase is well defined because
$\PBmap(\PTd(m_\cdot)(\conj{q})) = \PBmap(q_N) = m_N = \conj{m} =
\PBmap(\conj{q})$.

\begin{example}\label{ex:DC_mu-DC_form-PTd}
  The parallel transport operator associated to the discrete
  connection form $\DCl{\mu}$ defined in
  Example~\ref{ex:DC_mu-DC_form} is constructed as follows. For
  $\conj{r},\conj{r}'\in\R_{>0}$ fix a discrete path in $\R_{>0}$,
  $r_\cdot \in \Omega_N(\conj{r},\conj{r}')$ ---we ignore $\UDt''$
  because $\DCl{\mu}$ is globally defined--- and choose
  $\conj{q}:=(\conj{r},\conj{h}) \in \PBtotal|_{\conj{r}}$. Then, using
  $\HLdl{}{\mu}$ computed in Example~\ref{ex:DC_mu-HL}, the discrete
  lifted path of $r_\cdot$ starting at $\conj{q}$ is given by
  \begin{equation*}
    q_k =
    \begin{cases}
      \conj{q}, \stext{ if } k=0,\\
      (r_k,\conj{h}\, \exp(-i\sum_{j=1}^k(r_j-r_{j-1})^\mu)), \stext{
        if } k=1,\ldots, N.
    \end{cases}
  \end{equation*}
  Therefore
  $\PTd(\conj{r},\conj{h}) = (\conj{r}',(r_k,\conj{h}\,
  \exp(-i\sum_{j=1}^N(r_j-r_{j-1})^\mu))$ and
  \begin{equation*}
    \HPd(r_\cdot,(\conj{r},\conj{h})) =
    \exp\left(-i\sum_{j=1}^N(r_j-r_{j-1})^\mu\right).
  \end{equation*}
\end{example}

A natural question is how the discrete holonomy phase
$\HPd(m_\cdot,\conj{q})$ changes when $\conj{q}$ is replaced by
$l^\PBtotal_g(\conj{q})$.

\begin{prop}\label{prop:phase_and_change_in_fiber}
  Let $\HLd{}:\UDt'\rightarrow \PBtotal\times \PBtotal$ be a discrete
  horizontal lift on $\PBmap$, $\conj{m}\in \PBbase$,
  $\conj{q}\in \PBtotal|_{\conj{m}}$ and
  $m_\cdot\in \Omega_N(\conj{m})$. Then, for any $g\in G$,
  \begin{enumerate}
  \item\label{it:phase_and_change_in_fiber-path} If
    $q_\cdot\in \Omega_N(\conj{q},q_N)$ is the discrete horizontal
    lift of the path $m_\cdot$ starting at $\conj{q}$, the path
    $q'_\cdot$ defined by $q'_k:=l^\PBtotal_g(q_k)$ for $k=0,\ldots,N$ is the
    discrete horizontal lift of $m_\cdot$ starting at
    $l^\PBtotal_g(\conj{q})$.
  \item\label{it:phase_and_change_in_fiber-phase} In addition,
    $\HPd(m_\cdot,l^\PBtotal_g(\conj{q})) = g \HPd(m_\cdot,\conj{q})
    g^{-1}$.
  \end{enumerate}
\end{prop}

\begin{proof}
  By definition, $q'_\cdot\in
  \Omega_N(l^\PBtotal_g(\conj{q}),l^\PBtotal_g(q_N))$. That $q'_\cdot$ is the
  discrete horizontal lift of $m_\cdot$ starting at $l^\PBtotal_g(\conj{q})$
  follows by the definition of discrete horizontal lift and
  property~\eqref{it:HL_properties-HL_is_smooth} in
  Definition~\ref{def:HLd}. Thus,
  \begin{equation*}
    \begin{split}
      \HPd(m_\cdot,l^\PBtotal_g(\conj{q})) =&
      \kappa(l^\PBtotal_g(\conj{q}),\PTd(m_\cdot)(l^\PBtotal_g(\conj{q})))
      =
      \kappa(l^\PBtotal_g(\conj{q}),l^\PBtotal_g(\PTd(m_\cdot)(\conj{q})))
      \\=& g \kappa(\conj{q},\PTd(m_\cdot)(\conj{q})) g^{-1} =
      g\HPd(m_\cdot,\conj{q}) g^{-1}.
    \end{split}
  \end{equation*}
\end{proof}

The following two local results will allow us, later, to find explicit
formulas relating the holonomy phase and the curvature of a discrete
connection.

\begin{lemma}\label{le:local_expression_of_parallel_transport}
  Let $\DC:\UDt\rightarrow G$ be a discrete connection form on
  $\PBmap$, $V\subset \PBbase$ be an open subset and
  $s:V\rightarrow \PBtotal$ be a section of $\PBmap$. Then, for each
  $m_\cdot \in \Omega_{N,\VDt''}(m_0,m_N)$ (with $\VDt''$ as in
  Definition~\ref{def:DC_local_expression}), the following assertions
  are true.
  \begin{enumerate}
  \item\label{it:local_expression_of_parallel_transport-gs} Let
    $q_\cdot\in \Omega_N(q_0,q_N)$ be the discrete horizontal lift
    path of $m_\cdot$ starting at $q_0\in \PBtotal|_{m_0}$. Then,
    there are unique $g_k\in G$ so that, for $\varphi_s$ as in
    Lemma~\ref{le:properties_of_DC_s}, $\varphi_s(q_k) = (m_k,g_k)$
    for all $k=0,\ldots,N$.
  \item\label{it:local_expression_of_parallel_transport-pt} We have
    $\PTd(m_\cdot)(\varphi_s(m_0,g_0)) = \varphi_s(m_N,g_N)$ for
    $g_N:=g_0\prod_{k=1}^N \DCp{s}(m_{k-1},m_k)^{-1}$.
  \end{enumerate}
\end{lemma}

\begin{proof}
  As, $\PBmap(q_k) = m_k$ and
  $(m_{k-1},m_k)\in \VDt''\subset V\times V$, for all $k$, we see that
  $q_k\in
  \PBtotal|_V$. Point~\eqref{it:local_expression_of_parallel_transport-gs}
  follows immediately because $\varphi_s$ trivializes $\PBtotal|_{V}$.

  By definition of $\PTd$ and
  point~\eqref{it:local_expression_of_parallel_transport-gs} we have
  that $\PTd(m_\cdot)(\varphi_s(m_0,g_0)) = \varphi_s(m_N,g_N)$. Then,
  as $(q_{k-1},q_{k}) = h_{\DC}(q_{k-1},m_k)$ and
  $\varphi_s(m,g)=l^\PBtotal_g(s(m))$, recalling
  Remark~\ref{rem:HL_is_horizontal}, we have
  \begin{equation*}
    \begin{split}
      e =& \DC(q_{k-1},q_k) = \DC(\varphi_s(m_{k-1},g_{k-1}),
      \varphi_s(m_k,g_k)) \\=& \DC(l^\PBtotal_{g_{k-1}}(s(m_{k-1})),
      l^\PBtotal_{g_k}(s(m_k))) = g_k \DCp{s}(m_{k-1},m_k) g_{k-1}^{-1}.
    \end{split}
  \end{equation*}
  Thus, $g_k = g_{k-1} \DCp{s}(m_{k-1},m_k)^{-1}$ for all
  $k$. Point~\eqref{it:local_expression_of_parallel_transport-pt}
  follows by applying this formula recursively.
\end{proof}

\begin{prop}\label{prop:product_of_DC_and_product_of_BD}
  Let $\DC:\UDt\rightarrow G$ be a discrete connection on
  $\PBmap$. Then, for each $N \geq 2$, the following statements are
  true.
  \begin{enumerate}
  \item\label{it:product_of_DC_and_product_of_BD-no_section} For each
    $q_\cdot\in \Omega_N(\conj{m},\conj{m}')$ such that
    $(q_0,q_k,q_{k+1}) \in \UDt^{(3)}$ for all $k=1,\ldots,N-1$ we
    have
    \begin{equation*}
      \left( \prod_{k=1}^{N} \DC(q_{k-1},q_{k})^{-1}\right) \DC(q_0,q_{N}) =
      \prod_{k=1}^{N-1} \BD(q_0,q_k,q_{k+1})^{-1}.
    \end{equation*}
  \item\label{it:product_of_DC_and_product_of_BD-section} Let
    $V\subset \PBbase$ be an open subset and $s:V\rightarrow \PBtotal$
    be a section of $\PBtotal$. Then, for each
    $m_\cdot\in \Omega_{N}(\conj{m})$ such that
    $(m_0,m_k,m_{k+1}) \in {\VDt''}^{(3)}$
    (see~\eqref{eq:VDt''(3)-def}) for all $k=0,\ldots,N-1$ and
    $\conj{q}\in \PBtotal|_{\conj{m}}$ we have
    \begin{equation}\label{eq:product_of_DC_and_product_of_BD-phase_furmula}
      \HPd(m_\cdot,\conj{q}) =
      g_0\left(\prod_{k=1}^N \DCp{s}(m_{k-1},m_k)^{-1}\right) g_0^{-1} =
      \prod_{k=1}^{N-1} \BD(q_0,q_k,q_{k+1})^{-1},
    \end{equation}
    where $q_\cdot\in \Omega_N(\conj{q})$ is the discrete horizontal
    lift of $m_\cdot$ an $g_0$ satisfies
    $\varphi_s(m_0,g_0) = \conj{q}$.
  \end{enumerate}
\end{prop}

\begin{proof}
  Point~\eqref{it:product_of_DC_and_product_of_BD-no_section} can be
  readily checked by induction (on $N$). In order to prove
  point~\eqref{it:product_of_DC_and_product_of_BD-section} we notice
  that, since $(m_0,m_k,m_{k+1}) \in {\VDt''}^{(3)}$ for all
  $k=0,\ldots,N-1$, it follows that
  $m_\cdot\in \Omega_{N,\VDt''}(\conj{m})$. Hence, by
  Lemma~\ref{le:local_expression_of_parallel_transport} there exists
  the discrete horizontal lift $q_\cdot$ of $m_\cdot$ and, writing
  $q_k = \varphi_s(m_k,g_k)$ for all $k$, it satisfies
  $\PTd(m_\cdot)(\varphi_s(m_0,g_0)) =\varphi_s(m_N,g_N)$ for
  $g_N=g_0\prod_{k=1}^N \DCp{s}(m_{k-1},m_k)^{-1}$. Then,
  \begin{equation*}
    \begin{split}
      \HPd(m_\cdot,\conj{q}) =&
      \kappa(q_0,\PTd(m_\cdot)(\varphi_s(m_0,g_0))) =
      \kappa(\varphi_s(m_0,g_0),\varphi_s(\underbrace{m_N}_{=m_0},g_N))
      \\=& \kappa(l^\PBtotal_{g_0}(s(m_0)),l^\PBtotal_{g_N}(s(m_0))) =
      g_N g_0^{-1} = g_0\prod_{k=1}^N \DCp{s}(m_{k-1},m_k)^{-1}
      g_0^{-1},
    \end{split}
  \end{equation*}
  proving the first equality
  of~\eqref{eq:product_of_DC_and_product_of_BD-phase_furmula}. In
  order to prove the second equality, it is easy to check that, as
  $(m_0,m_k,m_{k+1})\in {\VDt''}^{(3)}$ for all $k$, we have
  $(s(m_0),s(m_k),s(m_{k+1}))\in {\UDt}^{(3)}$ for all $k$. Then,
  using the previous computation and the result of
  point~\eqref{it:product_of_DC_and_product_of_BD-no_section} we have
  \begin{equation*}
    \begin{split}
      \HPd(m_\cdot,\conj{q}) =& g_0\left(\prod_{k=1}^N
        \DCp{s}(m_{k-1},m_k)^{-1}\right) g_0^{-1} = g_0
      \left(\prod_{k=1}^N \DC(s(m_{k-1}),s(m_k))^{-1}\right) g_0^{-1}
      \\=& g_0 \left(\prod_{k=1}^{N-1}
        \BD(s(m_0),s(m_k),s(m_{k+1}))^{-1}\right)
      \underbrace{\DC(s(m_0),s(\overbrace{m_N}^{=m_0}))^{-1}}_{=e}
      g_0^{-1} \\=& g_0 \left(\prod_{k=1}^{N-1}
        \BDp{s}(m_0,m_k,m_{k+1})^{-1}\right) g_0^{-1}.
    \end{split}
  \end{equation*}
  The second equality
  of~\eqref{eq:product_of_DC_and_product_of_BD-phase_furmula} now
  follows from Lemma~\ref{le:properties_of_BD_s}.
\end{proof}

\begin{remark}
  By~\eqref{eq:product_of_DC_and_product_of_BD-phase_furmula}, we see
  that all the discrete holonomy phases (for loops satisfying the
  conditions of
  point~\eqref{it:product_of_DC_and_product_of_BD-section} of
  Proposition~\ref{prop:product_of_DC_and_product_of_BD}) are products
  of values of the (inverse of the) curvature of the discrete
  connection. In this sense, this result can be seen as a seminal
  observation towards a discrete analogue of the Ambrose--Singer
  Theorem.
\end{remark}


\section{Singular (co)homology and adaptations}
\label{sec:singular_cohomology_and_adaptations}

In this section we review some basic standard notions used in singular
homology theory (see, for
instance,~\cite{bo:munkres-elements_of_algebraic_topology}) and, then,
make an adaptation to have a theory that works with discrete
connection forms and curvatures.

Let $X$ be a topological space, $n\in\NZ$ and $\{e_0,\ldots,e_n\}$ be
the canonical basis of $\R^{n+1}$\footnote{In this context, it is
  convenient to consider $\R^n\subset \R^\N$ as the subset consisting
  of sequences vanishing after the $n$-th component. Thus, $\R^{n-1}$
  is a subspace of $\R^n$ and the inclusion a linear map.}.

\begin{definition}\label{def:singular_chains}
  Given $n\in\NZ$, the set
  \begin{equation*}
    \Delta_n:=\left\{\sum_{j=0}^n t_j e_j\in\R^{n+1}: t_j\geq 0 \text{ for }
    j=0,\ldots, n \text{ and } \sum_{j=0}^n t_j=1\right\}
  \end{equation*}
  is called the \jdef{standard $n$-simplex}. A \jdef{singular
    $n$-simplex} of $X$ is a continuous map
  $T_n:\Delta_n\rightarrow X$. A \jdef{singular $n$-chain} of $X$ is
  an element of the free abelian group generated by the singular
  $n$-simplexes, that we denote by $S_n(X)$. When $n\in\N$ and
  $k=0,\ldots,n$, the $k$-th face of the standard $n$-simplex is the
  map $\del_n^k:\Delta_{n-1}\rightarrow\Delta_{n}$ defined by
  \begin{gather*}
    \del_n^0(\sum_{j=0}^{n-1}t_je_j) := \sum_{j=1}^n t_{j-1} e_j,\\
    \del_n^k(\sum_{j=0}^{n-1}t_je_j) := \sum_{j=0}^{k-1} t_je_j +
    \sum_{j=k+1}^n t_{j-1}e_j, \stext{ if } k=1,\ldots,n.
  \end{gather*}
  For $n\in\N$ the \jdef{$n$-th boundary map} is the group
  homomorphism $\del_n:S_n(X)\rightarrow S_{n-1}(X)$ defined by
  \begin{equation*}
    \del_n(T):=\sum_{k=0}^n (-1)^k (T\circ \del_n^k).
  \end{equation*}
\end{definition}

We observe that, for $n\in \N$ and $e_j\in \Delta_{n-1}$, we have
\begin{equation*}
  \del_n^k(e_j) =
  \begin{cases}
    e_{j+1} \in \Delta_n \stext{ if } k=0,\ldots,j,\\
    e_j \in \Delta_n \stext{ if } k=j+1,\ldots,n.
  \end{cases}
\end{equation*}

\begin{example}\label{ex:singular_12_chains}
  Let $T_1$ be a singular $1$-simplex of $X$, then
  $\del_1T_1 \in S_0(X)$ satisfies
  $(\del_1 T_1)(1) = (T_1\circ\del_1^0)(1) - (T_1 \circ\del_1^1)(1) =
  T_1(0,1) - T_1(1,0)$. Similarly, if $T_2$ is a singular $2$-simplex
  of $X$, then $\del_2T_2 \in S_1(X)$ satisfies
  $(\del_2T_2)(t_0,t_1) = (T_2\circ \del_2^0)(t_0,t_1) - (T_2\circ
  \del_2^1)(t_0,t_1) + (T_2\circ \del_2^2)(t_0,t_1) = T_2(0,t_0,t_1) -
  T_2(t_0,0,t_1) + T_2(t_0,t_1,0)$.
\end{example}

\begin{definition}
  Let $A$ be an abelian group and $n\in\NZ$. The group
  $S^n(X,A):=\hom(S_n(X),A)$ is the group of \jdef{singular
    $n$-cochains} of $X$ and the group homomorphism
  $\delta_n:S^n(X,A)\rightarrow S^{n+1}(X,A)$ defined by
  $\delta_n\alpha_n := \alpha_n\circ \del_{n+1}$ is the \jdef{singular
    $n$-coboundary} of $X$.
\end{definition}

\begin{example}\label{ex:singular_12_cochains}
  Let $\alpha_0\in S^0(X,A)$ and $\alpha_1\in S^1(X,A)$; using the
  computations from Example~\ref{ex:singular_12_chains}, for any
  singular $1$-simplex $T_1$ of $X$ we have
  \begin{equation*}
    (\delta_0\alpha_0)(T_1) = \alpha_0(\del_1 T_1) =
    \alpha_0(T_1(0,1)-T_1(1,0)) = \alpha_0(T_1(0,1)) - \alpha_0(T_1(1,0)).
  \end{equation*}
  Analogously, for any singular $2$-simplex $T_2$ of $X$ we have,
  \begin{equation*}
    \begin{split}
      (\delta_1 \alpha_1)(T_2) =& \alpha_1((\del_2T_2)(t_0,t_1)) =
      \alpha_1(T_2(0,t_0,t_1) - T_2(t_0,0,t_1) + T_2(t_0,t_1,0)) \\=&
      \alpha_1(T_2(0,t_0,t_1)) - \alpha_1(T_2(t_0,0,t_1)) +
      \alpha_2(T_2(t_0,t_1,0)).
    \end{split}
  \end{equation*}
  Notice that in this example we denote the group operations in
  $S_n(X)$, $S^n(X,A)$ and $A$ additively. Later on we may use the
  product notation for the group operation in $A$ and $S^n(X,A)$.
\end{example}

\begin{remark}
  It is easy to check that $\del_{n-1}\circ \del_n = 0$ for all
  $n\geq 2$ and, consequently, that $\delta_{n+1}\circ \delta_n = 0$
  for all $n\geq 0$.
\end{remark}

\begin{definition}\label{def:integration_as_pairing}
  We denote the duality pairing between $S_n(X)$ and $S^n(X,A)$ by
  $\int$. Thus, for any $\alpha_n\in S^n(X,A)$ and $T_n \in S_n(X)$ we
  have
  \begin{equation*}
    \int_{T_n} \alpha_n := \alpha_n(T_n).
  \end{equation*}
\end{definition}

\begin{remark}
  For any $T_{n+1} \in S_{n+1}(X)$ and $\alpha_n \in S^n(X,A)$ we have
  \begin{equation}\label{eq:discrete_stokes_formula}
    \int_{\del_{n+1}T_{n+1}} \alpha_n = \alpha_n(\del_{n+1}T_{n+1}) =
    (\delta_n\alpha_n)(T_{n+1}) = \int_{T_{n+1}} \delta_n\alpha_n,
  \end{equation}
  that is, the ``singular Stokes' Theorem''.
\end{remark}


\subsection{Small singular chains and cochains}
\label{sec:small_singular_chains_and_cochains}

Next we want to consider a small variation of the previous
construction, with the goal of defining chains and cochains of $X$
that are, in a sense, controlled by an open set.

Let $X$ be a topological space and $U\subset X\times X$ be an open
subset (for the product topology). Then, for any $n\in\N$ we define
\begin{equation}\label{eq:multi_restriction-def}
  U^{(n+1)} := \{(x_0,x_1,\ldots,x_n)\in X^{n+1} :
  (x_j,x_k)\in U \text{ for all } 0\leq j<k\leq n\}.
\end{equation}

\begin{definition}\label{def:U_small_chain}
  For $n\in\N$, a singular $n$-simplex $T_n$ of $X$ is said to be
  \jdef{$U$-small} (\jdef{small} if no confusion arises) if
  $(T_n(e_0),\ldots, T_n(e_n)) \in U^{(n+1)}$. The free abelian group
  generated by the $U$-small singular $n$-simplexes of $X$ is denoted
  by $S_{n,U}(X)$ and its elements will be called \jdef{$U$-small
    singular $n$-chains} of $X$. For completeness, we define
  $S_{0,U}(X):=S_0(X)$.
\end{definition}

The following result is straightforward.
\begin{lemma}
  For each $n\geq 0$, $S_{n,U}(X)\subset S_n(X)$ is a subgroup and,
  for each $n\geq 1$, the map
  $\del_n^U:S_{n,U}(X)\rightarrow S_{n-1,U}(X)$ defined as the
  restriction and co-restriction of $\del_n$ is a homomorphism
  satisfying $\del_n^U\circ \del_{n+1}^U = 0$.
\end{lemma}

\begin{definition}\label{def:U_small_cochain}
  Let $A$ be an abelian group and $U\subset X\times X$ an open
  subset. For each $n\in\NZ$ we define the group
  $S^{n,U}(X,A):=\hom(S_{n,U}(X),A)$; its elements are called
  \jdef{$U$-small singular cochains of $X$}. The group homomorphism
  $\delta_n^U:S^{n,U}(X,A)\rightarrow S^{n+1,U}(X,A)$ defined by
  $\delta_n^U \alpha_n := \alpha_n \circ \del_{n+1}^U$ is the
  \jdef{small singular $n$-coboundary} of $X$.
\end{definition}

The following result is straightforward.
\begin{lemma}\label{le:properties_of_small_cochains}
  For each $n\in\NZ$, $S^{n}(X,A)\subset S^{n,U}(X,A)$ is a subgroup
  and the homomorphism $\delta_n^U$ restricted and co-restricted
  appropriately coincides with $\delta_n$. Furthermore,
  $\delta_{n+1}^U\circ \delta_n^U = 0$.
\end{lemma}

Last we have a simple ``change of coefficients'' relation. Let
$f\in\hom(A,A')$, where $A$ and $A'$ are abelian groups. It is easy to
check that the map $f_*:S^{n,U}(X,A)\rightarrow S^{n,U}(X,A')$ defined
by $f_*(\alpha_n):=f\circ \alpha_n$ is a homomorphism satisfying
$f_* \circ \delta_n^{U,A} = \delta_{n+1}^{U,A'}\circ f_*$ (\ie, $f_*$
is a homomorphism of cochain complexes).

As an application of the ideas developed so far, in the next section
we construct small singular cochains associated to a discrete
connection on $\PBmap$, when the structure group of the principal bundle
is abelian.


\subsection{Discrete connections and singular cochains}
\label{sec:discrete_connections_and_singular_cochains}

Let $\DC:\UDt\rightarrow G$ be a discrete connection on $\PBmap$ and
assume that the structure group $G$ of $\PBmap$ is abelian. Let
$T_1\in S_{1,\UDt}(\PBtotal)$ be a $1$-simplex. Then, we define
\begin{equation*}
  [\DC](T_1) := \DC(T_1(1,0),T_1(0,1)) \in G.
\end{equation*}
Being these $1$-simplexes free generators of $S_{1,\UDt}(\PBtotal)$, this
assignment defines a unique element $[\DC]\in S^{1,\UDt}(\PBtotal,G)$.

Observe that, if $\PBtotal$ is connected, the cochain $[\DC]$
determines the function $\DC:\UDt\rightarrow G$. Thus, the discrete
connection form of a discrete connection is a true $1$-form (a
cochain) for this (small) singular cohomology, at least when $G$ is
abelian.  If $\BD:\UDt^{(3)}\rightarrow G$ is the curvature of $\DC$
and $T_2\in S_{2,\UDt}(\PBtotal,G)$ is a $2$-simplex, we define
\begin{equation*}
  [\BD](T_2):=\BD(T_2(1,0,0),T_2(0,1,0),T_2(0,0,1)) \in G.
\end{equation*}
As those simplexes freely generate $S_{2,\UDt}(\PBtotal)$, this formula
determines a unique $[\BD]\in S^{2,\UDt}(\PBtotal,G)$.

\begin{example}\label{ex:DC_mu-cochains}
  In the case of the discrete connection form $\DCl{\mu}$ introduced
  in Example~\ref{ex:DC_mu-DC_form}, if
  $T_1 = \sum_{j=1}^N a_j T_1^j \in S_{1,\PBtotal\times
    \PBtotal}(\R_{>0}\times U(1)) = S_{1}(\R_{>0}\times U(1))$, using
  multiplicative notation, we have
  \begin{equation*}
    [\DCl{\mu}](T_1) = \prod_{j=1}^N
    (\DCl{\mu}(\underbrace{T_1^j(1,0)}_{=:(r_0^j,h_0^j)},
    \underbrace{T_1^j(0,1)}_{=:(r_1^j,h_1^j)})^{a_j}
    = \exp\left(i \sum_{j=1}^N a_j (r_1^j-r_0^j)^\mu\right)
    \prod_{j=1}^N \left(\frac{h_1^j}{h_0^j}\right)^{a_j} .
  \end{equation*}
  Similarly, if
  $T_2 = \sum_{j=1}^N a_j T_2^j \in S_{2,\PBtotal\times
    \PBtotal}(\R_{>0}\times U(1)) = S_{2}(\R_{>0}\times U(1))$, we
  have
  \begin{equation*}
    \begin{split}
      [\BDl{\mu}](T_2) =& \prod_{j=1}^N
      (\BDl{\mu}(\underbrace{T_2^j(1,0,0)}_{=:(r_0^j,h_0^j)},
      \underbrace{T_2^j(0,1,0)}_{=:(r_1^j,h_1^j)},
      \underbrace{T_2^j(0,0,1)}_{=:(r_2^j,h_2^j)})^{a_j} \\=&
      \exp\left( i \sum_{j=1}^N
        a_j\left(-(r_2^j-r_0^j)^\mu+(r_2^j-r_1^j)^\mu +
          (r_1^j-r_0^j)^\mu \right)\right).
    \end{split}
  \end{equation*}
\end{example}

\begin{prop}\label{prop:dDC=BD-G}
  For $[\DC]\in S^{1,\UDt}(\PBtotal,G)$ and
  $[\BD]\in S^{2,\UDt}(\PBtotal,G)$ as above, we have
  $[\BD] = \delta_1^\UDt [\DC]$ and $\delta_2^\UDt [\BD]=e$.
\end{prop}

\begin{proof}
  The second assertion follows from the first and the fact that
  $\delta_2^\UDt \circ \delta_1^\UDt = e$. As the $2$-simplexes
  generate $S_{2,\UDt}(\PBtotal)$, it suffices to check the first
  assertion of the statement on them to conclude its validity in
  general. Let $T_2$ be one such $2$-simplex. Then,
  \begin{equation*}
    \begin{split}
      [\BD](T_2) =& \BD(T_2(1,0,0),T_2(0,1,0),T_2(0,0,1)) \\=&
      \DC(T_2(1,0,0),T_2(0,0,1))^{-1}
      \DC(T_2(0,1,0),T_2(0,0,1))\DC(T_2(1,0,0),T_2(0,1,0)).
    \end{split}
  \end{equation*}
  On the other hand, converting to multiplicative notation the
  computations of Example~\ref{ex:singular_12_cochains},
  \begin{equation*}
    \begin{split}
      (\delta_1^\UDt [\DC])(&T_2) = [\DC](T_2(0,t_0,t_1))
      ([\DC](T_2(t_0,0,t_1)))^{-1}[\DC](T_2(t_0,t_1,0)) \\=&
      \DC(T_2(0,1,0),T_2(0,0,1)) \DC(T_2(1,0,0),T_2(0,0,1))^{-1}
      \DC(T_2(1,0,0),T_2(0,1,0)).
    \end{split}
  \end{equation*}
  The result follows by comparing the two formulas.
\end{proof}

\begin{remark}
  The expression $\delta_2^\UDt [\BD]=e$ may be interpreted as a
  discrete version of Bianchi's Identity.
\end{remark}

A nice immediate consequence of the formalism developed is the
following result.

\begin{corollary}\label{cor:DC_exact_imp_flat}
  Let $\DC:\UDt\rightarrow G$ be a discrete connection on
  $\PBmap:\PBtotal\rightarrow \PBbase$ with $\PBtotal$ connected. If
  there is $\alpha_0\in S^0(\PBtotal)$ such that
  $[\DC] = \delta_0^\UDt \alpha_0$ then $\DC$ is flat, \ie, $\BD=e$.
\end{corollary}

\begin{proof}
  As $[\DC] = \delta_0^\UDt \alpha_0$, by
  Proposition~\ref{prop:dDC=BD-G} and
  Lemma~\ref{le:properties_of_small_cochains}, we have
  $[\BD] = \delta_1^\UDt [\DC] = (\delta_1^\UDt \circ \delta_0^\UDt)
  \alpha_0 = e$. Hence, for any $T_2\in S_{2,\UDt}(\PBtotal)$,
  \begin{equation}\label{eq:DC_exact_imp_flat-BD=0}
    e = [\BD](T_2) = \BD(T_2(1,0,0),T_2(0,1,0),T_2(0,0,1)).
  \end{equation}
  For any $(q_0,q_1,q_2)\in\UDt^{(3)}$, as $\PBtotal$ is connected
  (hence, path connected), there is a continuous map
  $\gamma:[0,1]\rightarrow \PBtotal$ such that $\gamma(0)=q_0$,
  $\gamma(\frac{1}{2})=q_1$ and $\gamma(1)=q_2$. Let
  $T_2:\Delta_2\rightarrow \PBtotal$ be defined by
  $T_2(t_0,t_1,t_2):=\gamma(\frac{1}{2}t_1+t_2)$; it is immediate that
  $T_2\in S_{2,\UDt}(\PBtotal)$ and it satisfies $T_2(e_j)=q_j$ for
  $j=0,1,2$. All together, using~\eqref{eq:DC_exact_imp_flat-BD=0},
  \begin{equation*}
    \BD(q_0,q_1,q_2) = \BD(T_2(1,0,0),T_2(0,1,0),T_2(0,0,1)) = e.
  \end{equation*}
\end{proof}


\subsection{The local version}
\label{sec:the_local_version}

Just as we described in
Section~\ref{sec:discrete_connections_and_singular_cochains} how to
view discrete connection forms and the corresponding curvature of a
principal bundle as (small) singular cochains, it is possible to do
the same for the local expressions of those objects, as we discuss
below.

Let $V\subset \PBbase$ be an open subset, $s:V\rightarrow \PBtotal$ be
a section of $\PBmap$ and $\DC:\UDt\rightarrow G$ be a discrete
connection form on $\PBmap$. Let $\DCp{s}:\VDt''\rightarrow G$ be the
corresponding local expression of $\DC$ with respect to
$s$~\eqref{eq:DC-local_expression-def}. For each $1$-simplex
$T_1\in S_{1,\VDt''}(\PBbase)$ we define
\begin{equation*}
  [\DCp{s}](T_1) := \DCp{s}(T_1(1,0),T_1(0,1)) \in G.
\end{equation*}
Being those simplexes generators of $S_{1,\VDt''}(\PBbase)$, this
assignment determines a unique element
$[\DCp{s}]\in S^{1,\VDt''}(\PBbase,G)$.

Similarly, let $\BDp{s}:{\VDt''}^{(3)}\rightarrow G$ be the local
expression of the curvature of
$\DC$~\eqref{eq:BD-local_expression-def} and $T_2\in S_{2,\VDt''}(\PBbase)$
be a $2$-simplex. We define
\begin{equation*}
  [\BDp{s}](T_2) := \BDp{s}(T_2(1,0,0),T_2(0,1,0),T_2(0,0,1)) \in G.
\end{equation*}
Just as before, as the $2$-simplexes generate $S_{2,\VDt''}(\PBbase)$ this
last expression defines a unique element
$[\BDp{s}]\in S^{2,\VDt''}(\PBbase,G)$.

\begin{example}\label{ex:DC_mu-local_cochains}
  In the case of the discrete connection form $\DCl{\mu}$ introduced
  in Example~\ref{ex:DC_mu-DC_form} together with the (global)
  trivialization of Example~\ref{ex:DC_mu-DC_form-local}, if
  $T_1 = \sum_{j=1}^N a_j T_1^j \in S_{1,\PBbase\times
    \PBbase}(\R_{>0}) = S_{1}(\R_{>0})$, we have
  \begin{equation*}
    [\DClp{\mu}{s}](T_1) = \prod_{j=1}^N
    (\DClp{\mu}{s}(\underbrace{T_1^j(1,0)}_{=:r_0^j},
    \underbrace{T_1^j(0,1)}_{=:r_1^j})^{a_j}
    = \exp\left(i \sum_{j=1}^N a_j (r_1^j-r_0^j)^\mu\right).
  \end{equation*}
  Similarly, if
  $T_2 = \sum_{j=1}^N a_j T_2^j \in S_{2,\PBbase\times \PBbase}(\R_{>0}) =
  S_{2}(\R_{>0})$, we have
  \begin{equation*}
    \begin{split}
      [\BDlp{\mu}{s}](T_2) =& \prod_{j=1}^N
      (\BDlp{\mu}{s}(\underbrace{T_2^j(1,0,0)}_{=:r_0^j},
      \underbrace{T_2^j(0,1,0)}_{=:r_1^j},
      \underbrace{T_2^j(0,0,1)}_{=:r_2^j})^{a_j} \\=& \exp\left( i
        \sum_{j=1}^N a_j\left(-(r_2^j-r_0^j)^\mu+(r_2^j-r_1^j)^\mu +
          (r_1^j-r_0^j)^\mu \right)\right).
    \end{split}
  \end{equation*}
\end{example}

It is easy to prove the following local analogue of
Lemma~\ref{le:dDC=BD-G}.

\begin{lemma}\label{le:dDC=BD-G}
  For $[\DCp{s}]\in S^{1,\VDt''}(\PBbase,G)$ and
  $[\BDp{s}]\in S^{2,\UDt''}(\PBbase,G)$ as above, we have
  $[\BDp{s}] = \delta_1^{\VDt''} [\DCp{s}]$ and
  $\delta_2^{\VDt''} [\BDp{s}]=e$.
\end{lemma}


\section{Holonomy around a loop}
\label{sec:holonomy_around_a_loop}

In this section we find an explicit formula for the discrete holonomy
phase around a loop that is contained in an open set trivializing
$\PBmap$. We still work in the case where $G$, the structure group of
$\PBmap$, is abelian.

\begin{remark}
  By Proposition~\ref{def:holonomy_around_a_loop} the different values
  of the discrete holonomy phase around a loop, starting at different
  points of the corresponding fiber are conjugated elements of
  $G$. When $G$ is abelian, these elements are all the same, so that
  the discrete phase around a loop only depends on the loop and we
  denote it by $\HPd(m_\cdot)$.
\end{remark}

\begin{lemma}\label{le:paths_and_simplexes}
  Let $\DC:\UDt\rightarrow G$ be a discrete connection on $\PBmap$,
  whose structure group $G$ is abelian. Let $V\subset \PBbase$ be a
  connected open subset and $s:\PBbase\rightarrow \PBtotal$ be a local
  section of $\PBmap$. For any discrete loop in $\PBbase$,
  $m_\cdot\in\Omega_{N,\VDt''}(\conj{m})$, we have that
  \begin{enumerate}
  \item\label{it:paths_and_simplexes-existence} there exists
    $\ti{m} = \sum_{k=1}^{N} T_1^k \in S_{1,\VDt''}(\PBbase)$ such that
    $T_1^k(e_0) = m_{k-1}$ and $T_1^k(e_1)=m_{k}$ for $k=1,\ldots,N$.
  \item\label{it:paths_and_simplexes-integral} Then, for any such
    $\ti{m}$,
    $\prod_{k=1}^N \DCp{s}(m_{k-1},m_{k})^{-1} = \left(\int_{\ti{m}}
      [\DCp{s}]\right)^{-1}$.
  \end{enumerate}
\end{lemma}

\begin{proof}
  Being $V$ connected, for each $k=1,\ldots,N$ there is a continuous
  path $\gamma_{k-1,k}:[0,1]\rightarrow V$ such that
  $\gamma_{k-1,k}(0) = m_{k-1}$ and $\gamma_{k-1,k}(1) =
  m_{k}$. Define $T_1^k:\Delta_1\rightarrow \PBbase$ by
  $T_1^k((1-t) e_0 + t e_1) := \gamma_{k-1,k}(t)$ for $t\in [0,1]$. As
  $(T_1^k(e_i),T_1^k(e_j)) \in \VDt''$ for all $i,j=0,1$, we see that
  $\ti{m} := \sum_{k}^{N} T_1^k \in S_{1,\VDt''}(\PBbase)$, proving
  point~\eqref{it:paths_and_simplexes-existence}.  Then, by definition
  of $\int$,
  \begin{equation*}
    \begin{split}
      \int_{\ti{m}} [\DCp{s}] =& \DCp{s}(\ti{m}) =
      \DCp{s}\left(\sum_{k=1}^{N} T_1^k\right) = \prod_{k=1}^{N}
      \DCp{s}(T_1^k) \\=& \prod_{k=1}^{N}
      \DCp{s}(T_1^k(e_0),T_1^k(e_1)) = \prod_{k=1}^{N}
      \DCp{s}(m_{k-1},m_{k}),
    \end{split}
  \end{equation*}
  that, on inversion, leads to
  point~\eqref{it:paths_and_simplexes-integral}.
\end{proof}

We say that the cochain $\ti{m}\in S_{1,\VDt''}(\PBbase)$ satisfying
point~\eqref{it:paths_and_simplexes-existence} of
Lemma~\ref{le:paths_and_simplexes} \jdef{interpolates} the discrete
path $m_\cdot$

\begin{theorem}\label{th:phases_and_loops-group}
  Let $\DC:\UDt\rightarrow G$ be a discrete connection on the
  principal $G$-bundle $\PBmap:\PBtotal\rightarrow \PBbase$ with $G$
  abelian, $V\subset \PBbase$ a connected open subset and
  $s:V\rightarrow \PBtotal$ a local section of $\PBmap$. For any
  $m_\cdot\in \Omega_N(\conj{m})$ such that
  $(m_0,m_k,m_{k+1}) \in {\VDt''}^{(3)}$ for all $k=0,\ldots,N-1$, we
  have
  \begin{equation}\label{eq:phases_and_loops-group-DC}
    \HPd(m_\cdot) = \left(\int_{\ti{m}}[\DCp{s}]\right)^{-1},
  \end{equation}
  where $\ti{m}\in S_{1,\VDt''}(\PBbase)$ is any small singular chain
  interpolating the discrete path $m_\cdot$, in the sense of
  Lemma~\ref{le:paths_and_simplexes}. If, in addition, there is
  $\ti{\sigma}\in S_{2,\VDt''}(\PBbase)$ so that
  $\del_2^{\VDt''}(\ti{\sigma}) = \ti{m}$, then
  \begin{equation}\label{eq:phases_and_loops-group-BD}
    \HPd(m_\cdot) = \left(\int_{\ti{\sigma}}[\BDp{s}]\right)^{-1}.
  \end{equation}
\end{theorem}

\begin{proof}
  Identity~\eqref{eq:phases_and_loops-group-DC} follows from the first
  equality in~\eqref{eq:product_of_DC_and_product_of_BD-phase_furmula}
  taking into account that $G$ is abelian and, then,
  point~\eqref{it:paths_and_simplexes-integral} of
  Lemma~\ref{le:paths_and_simplexes}. On the other hand, if
  $\ti{\sigma}\in S_{2,\VDt''}(\PBbase)$ satisfies
  $\del_2^{\VDt''}(\ti{\sigma}) = \ti{m}$, using
  Proposition~\ref{le:dDC=BD-G} and (a ``small version''
  of)~\eqref{eq:discrete_stokes_formula}, we have
  \begin{equation*}
    \int_{\ti{\sigma}} [\BDp{s}] =
    \int_{\ti{\sigma}}\delta_1^{\VDt''}[\DCp{s}]
    = \int_{\del_2^{\VDt''}\ti{\sigma}} [\DCp{s}] =
    \int_{\ti{m}}[\DCp{s}],
  \end{equation*}
  and~\eqref{eq:phases_and_loops-group-BD} follows
  from~\eqref{eq:phases_and_loops-group-DC}.
\end{proof}

\begin{example}\label{ex:DC_mu-DC_form-integrals_group}
  In the case of the discrete connection form $\DCl{\mu}$ introduced
  in Example~\ref{ex:DC_mu-DC_form} together with the (global)
  trivialization of Example~\ref{ex:DC_mu-DC_form-local}, if
  $r_\cdot\in \Omega_{N,\VDt''}(\conj{r}) = \Omega_N(\conj{r})$, it
  automatically satisfies that
  $(r_0,r_k,r_{k+1}) \in {\VDt''}^{(3)} = \R_{>0}^3$ for all $k$, so
  that we can apply Theorem~\ref{th:phases_and_loops-group} to compute
  the discrete holonomy phase $\HPd(r_\cdot)$. A $1$-chain in
  $\R_{>0}$ that interpolates $r_\cdot$ is
  $\ti{r} := \sum_{j=1}^N T_1^j$ for
  $T_1^j(t_0e_0+t_1e_1) := t_0 r_{j-1}+t_1 r_j $, so that
  $\ti{r}\in S_1(\R_{>0})$.  Thus,
  using~\eqref{eq:phases_and_loops-group-DC} and
  Example~\ref{ex:DC_mu-local_cochains}, we have
  \begin{equation}\label{eq:DC_mu-DC_form-integrals_group-DC}
    \HPd(r_\cdot) = \left( \int_{\ti{r}} [\DCl{\mu}] \right)^{-1} =
    ([\DCl{\mu}](\ti{r}))^{-1} =
    \exp\left(-i \sum_{j=1}^N (r_j-r_{j-1})^\mu\right).
  \end{equation}
  We construct $\ti{\sigma}\in S_2(\R_{>0})$ as follows. First, for
  each $j=1,\ldots,N-1$, define
  \begin{equation*}
    \gamma_j(t):=
    \begin{cases}
      3t r_j + (1-3t) r_0,\stext{ if } 0\leq t\leq\frac{1}{3},\\
      (3t-1) r_{j+1} + (2-3t) r_j, \stext{ if }
      \frac{1}{3}\leq t\leq \frac{2}{3},\\
      (3t-2) r_0 + (3-3t) r_{j+1}, \stext{ if } \frac{2}{3}\leq t \leq
      1,
    \end{cases}
  \end{equation*}
  and
  $T_2^j(t_0e_0+t_1e_1+t_2e_2) :=
  \gamma_j(\frac{1}{3}t_1+\frac{2}{3}t_2)$. Finally,
  $\ti{\sigma}:=\sum_{j=1}^{N-1} T_2^j$. By construction,
  $\ti{\sigma}\in S_{2}(\R_{>0})$ and it is not hard to check that
  $\del_2\ti{\sigma} = \ti{r}$. Then,
  using~\eqref{eq:phases_and_loops-group-BD}
  and~\eqref{eq:DC_mu-curvature-local_expression},
  \begin{equation*}
    \begin{split}
      \HPd(m_\cdot) =& \left(\int_{\ti{\sigma}}[\BDlp{\mu}{s}]\right)^{-1} =
      \left( [\BDlp{\mu}{s}](\ti{\sigma}) \right)^{-1} \\=&
      \prod_{j=1}^{N-1}
      (\BDlp{\mu}{s}(T_2^j(1,0,0),T_2^j(0,1,0),T_2^j(0,0,1)))^{-1}
      = \prod_{j=1}^{N-1} (\BDlp{\mu}{s}(r_0,r_j,r_{j+1}))^{-1}
      \\=& \prod_{j=1}^{N-1} \exp(-i(-(r_{j+1}-r_0)^\mu +
      (r_{j+1}-r_j)^\mu + (r_j-r_0)^\mu)) \\=&
      \exp\left(-i\sum_{j=1}^{N-1}(-(r_{j+1}-r_0)^\mu +
        (r_{j+1}-r_j)^\mu + (r_j-r_0)^\mu)\right) \\=&
      \exp\left(-i\sum_{j=0}^{N-1} (r_{j+1}-r_j)^\mu\right),
    \end{split}
  \end{equation*}
  matching the previous computation of $\HPd(m_\cdot)$.
\end{example}

\begin{remark}
  For (continuous) connections on a principal $G$-bundle the set of
  all possible holonomy phases starting at a given point is known as
  the holonomy group of the connection and it encodes interesting
  information about the connection and the bundle (see, for
  instance,~\cite{bo:kobayashi_nomizu-foundations-v1}). A similar set
  can be constructed in the discrete setting:
  $\Hold(\conj{q}) := \{\HPd(m_\cdot,\conj{q}) \in G : m_\cdot\in
  \Omega_{\UDt''}(\PBmap(\conj{q}))\}$. As a subset of the group $G$,
  $\Hold(\conj{q})$ is, in general, a submonoid (\ie, it contains the
  identity element and is closed under products). If, in addition, the
  discrete connection is \jdef{symmetric} (essentially,
  $\DC(q_0,q_1) = \DC(q_1,q_0)^{-1}$ for all $(q_0,q_1)\in \UDt$),
  $\Hold(\conj{q})$ is a subgroup of $G$
  (see~\cite{ar:fernandez_zuccali-holonomy_of_discrete_connections}
  for more details on discrete holonomy).

  For the discrete connection $\DCl{\mu}$ introduced in
  Example~\ref{ex:DC_mu-DC_form}, it is easy to see,
  using~\eqref{eq:DC_mu-DC_form-integrals_group-DC}, that
  $\Hold(\conj{q}) = U(1)$ for all $\mu>1$ while
  $\Hold(\conj{q}) = \{1\}$ if $\mu=1$.
\end{remark}

Notice that identities~\eqref{eq:phases_and_loops-group-DC}
and~\eqref{eq:phases_and_loops-group-BD} are discrete analogues
of~\eqref{eq:phase_in_terms_of_connection-cont}
and~\eqref{eq:phase_in_terms_of_curvature-cont}, although not
``exponentiated''. In the next section we make a few formal changes to
consider singular cochains with values in $\jgg:=\lie{G}$, so that we
can obtain exact discrete analogues
of~\eqref{eq:phase_in_terms_of_connection-cont}
and~\eqref{eq:phase_in_terms_of_curvature-cont}.


\section{Forms with values in the Lie algebra}
\label{sec:forms_with_values_in_the_lie_algebra}

The goal of this section is to define functions $\LDCp{s}$ and
$\LBDp{s}$ with values in $\jgg$ so that $\DCp{s} = \exp_G(\LDCp{s})$
and $\BDp{s} = \exp_G(\LBDp{s})$. Then, use them to define singular
cochains that will enable us to prove discrete analogues
of~\eqref{eq:phase_in_terms_of_connection-cont}
and~\eqref{eq:phase_in_terms_of_curvature-cont}.


\subsection{Logarithms of $\DCp{s}$ and $\BDp{s}$}
\label{sec:logarithms_of_DCs_and_BDs}

The main tool that we need is the following classical result.

\begin{theorem}\label{thm:exp_is_diffeo}
  Let $G$ be a Lie group with Lie algebra $\jgg$. Then, the
  exponential map $\exp_G:\jgg\rightarrow G$ is a diffeomorphism
  between open neighborhoods $U_0$ and $V_e$ of $0\in \jgg$ and
  $e\in G$. In addition, when $G$ is abelian, $\exp$ is a homomorphism
  of groups, considering $\jgg$ as a Lie group with the operation
  given by the addition.
\end{theorem}

\begin{proof}
  See Theorem 2.10.1 and Corollary 2.13.3
  of~\cite{bo:varadarajan-lie_groups_lie_algebras_and_their_representations}.
\end{proof}

The following result uses Theorem~\ref{thm:exp_is_diffeo} to obtain
some additional open sets.

\begin{lemma}\label{le:exp_is_diffeo_and_more}
  With the same notation of Theorem~\ref{thm:exp_is_diffeo}, there are
  open neighborhoods $U_0'\subset U_0$ and $V_e'\subset V_e$ of
  $0\in \jgg$ and $e\in G$ such that
  $\exp_G|_{U_0'}:U_0'\rightarrow V_e'$ is a diffeomorphism and the
  following conditions hold.
  \begin{enumerate}
  \item \label{it:exp_is_diffeo_and_more-inverse} $U_0'$ and $V_e'$ are
    invariant under $a\mapsto -a$ and $g\mapsto g^{-1}$ respectively.
  \item \label{it:exp_is_diffeo_and_more-product} For any
    $a_0,a_1,a_2\in U_0'$, $a_0+a_1+a_2\in U_0$ and, when $G$ is
    abelian, for any $g_0,g_1,g_2\in V_e'$, we have
    $g_0g_1g_2\in V_e$.
  \end{enumerate}
\end{lemma}

\begin{proof}
  It is a standard application of topological properties, mostly
  continuity of the group and algebra operations.
\end{proof}

In what follows we assume that the structure group $G$ of $\PBmap$ is
abelian and that $\DC:\UDt\rightarrow G$ is a discrete connection form
on $\PBmap$. Also, we fix an open subset $V\subset \PBbase$ and a
section $s:V\rightarrow \PBtotal$ of $\PBmap$.

\begin{definition}\label{def:adt}
  Let $\WDt'':= (\DCp{s})^{-1}(V_e') \subset \PBbase\times \PBbase$ and
  $\LDCp{s}:\WDt''\rightarrow U_0'$ so that
  $\exp_G\circ \LDCp{s} = \DCp{s}$. We call $\LDCp{s}$ the
  \jdef{logarithm of the local expression of $\DC$}.
\end{definition}

\begin{lemma}\label{le:adt_is_well_defined}
  In the context of Definition~\ref{def:adt}, $\WDt''\subset\VDt''$ is
  open, and $\LDCp{s}$ is well defined and smooth.
\end{lemma}

\begin{proof}
  As $\DCp{s}:\VDt''\rightarrow G$ is smooth, and $V_e'\subset G$ is
  open, $\WDt''\subset \VDt''$ is open. For any $(m_0,m_1)\in \WDt''$,
  we have $\DCp{s}(m_0,m_1) \in V_e'$ and, as
  $\exp_G|_{U_0'}:U_0'\rightarrow V_e'$ is a diffeomorphism, there is
  a unique $\LDCp{s}(m_0,m_1)\in U_0'$ such that
  $\exp_G(\LDCp{s}(m_0,m_1)) = \DCp{s}(m_0,m_1)$. The smoothness of
  $\LDCp{s}$ follows from that of $\DCp{s}$ and the fact that
  $\exp_G|_{U_0'}$ is a diffeomorphism.
\end{proof}

In the same spirit, we have the following notion.

\begin{definition}\label{def:bdt}
  Let $\ti{\WDt}'':= (\BDp{s})^{-1}(V_e) \subset \PBbase^3$ and
  $\LBDp{s}:\ti{\WDt}''\rightarrow U_0$ so that
  $\exp_G\circ \LBDp{s} = \BDp{s}$. We call $\LBDp{s}$ the
  \jdef{logarithm of the local expression of $\BD$}.
\end{definition}

\begin{lemma}\label{le:bdt_is_well_defined}
  In the context of Definition~\ref{def:bdt}, we have
  \begin{enumerate}
  \item \label{it:bdt_is_well_defined-open}
    $\ti{\WDt}''\subset {\VDt''}^{(3)}$ is open and
    ${\WDt''}^{(3)}\subset \ti{\WDt}''$, where $\VDt''$ is introduced
    in Definition~\ref{def:DC_local_expression}, $\WDt''$ in
    Definition~\ref{def:adt} and ${\VDt''}^{(3)}, {\WDt''}^{(3)}$
    follow~\eqref{eq:multi_restriction-def}.
  \item \label{it:bdt_is_well_defined-well} $\LBDp{s}$ is well defined
    and smooth.
  \item \label{it:bdt_is_well_defined-relation_ldc}
    $\LBDp{s}(m_0,m_1,m_2) =
    -\LDCp{s}(m_0,m_2)+\LDCp{s}(m_1,m_2)+\LDCp{s}(m_0,m_1)$ for all
    $(m_0,m_1,m_2)\in {\WDt''}^{(3)}$.
  \end{enumerate}
\end{lemma}

\begin{proof}
  As $\BDp{s}:{\VDt''}^{(3)}\rightarrow G$,
  $\ti{\WDt}''= \BDp{s}^{-1}(V_e')\subset {\VDt''}^{(3)}$ and its
  openness follows from that of $V_e'$ and the continuity of
  $\BDp{s}$. On the other hand, for $(m_0,m_1,m_2)\in {\WDt''}^{(3)}$,
  we have $(m_j,m_k)\in \WDt''$ for all $j,k=0,1,2$ and, so,
  $\DCp{s}(m_j,m_k)\in V_e'$ for all $j,k=0,1,2$. Then, by
  Lemma~\ref{le:properties_of_BD_s} and
  point~\eqref{it:exp_is_diffeo_and_more-product} of
  Lemma~\ref{le:exp_is_diffeo_and_more}, we have
  \begin{equation*}
    \BDp{s}(m_0,m_1,m_2) = \DCp{s}(m_0,m_2)^{-1} \DCp{s}(m_1,m_2)
    \DCp{s}(m_0,m_1) \in V_e.
  \end{equation*}
  Thus, $(m_0,m_1,m_2)\in(\BDp{s})^{-1}(V_e)$, proving
  point~\eqref{it:bdt_is_well_defined-open}. As
  $\BDp{s}(\ti{\WDt}'')\subset V_e$ and
  $\exp_G|_{U_0}:U_0\rightarrow V_e$ is a diffeomorphism,
  $\LBDp{s} = (\exp_G|_{U_0})^{-1}\circ \BDp{s}|_{\ti{\WDt}''}$ is
  well defined and smooth, proving
  point~\eqref{it:bdt_is_well_defined-well}. In order to prove
  point~\eqref{it:bdt_is_well_defined-relation_ldc} we observe that,
  for $(m_0,m_1,m_2)\in {\WDt''}^{(3)}$,
  \begin{equation}\label{eq:bdt_is_well_defined-relation_ldc-exp_id}
    \begin{split}
      \exp_G(\LBDp{s}(m_0,m_1,m_2)) =& \BDp{s}(m_0,m_1,m_2) =
      \DCp{s}(m_0,m_2)^{-1} \DCp{s}(m_1,m_2) \DCp{s}(m_0,m_1) \\=&
      (\exp_G(\LDCp{s}(m_0,m_2)))^{-1} \exp_G(\LDCp{s}(m_1,m_2))
      \exp_G(\LDCp{s}(m_0,m_1)) \\=&
      \exp_G(-\LDCp{s}(m_0,m_2)+\LDCp{s}(m_1,m_2) + \LDCp{s}(m_0,m_1)),
    \end{split}
  \end{equation}
  where the second equality is by Lemma~\ref{le:properties_of_BD_s},
  the third is because $(m_j,m_k)\in \WDt''$ for all $j,k$ and the
  last by Theorem~\ref{thm:exp_is_diffeo} and the commutativity of
  $G$. Then, as $\LDCp{s}(m_j,m_k)\in U_0'$, by
  Lemma~\ref{le:exp_is_diffeo_and_more},
  $-\LDCp{s}(m_0,m_2)+\LDCp{s}(m_1,m_2) + \LDCp{s}(m_0,m_1) \in
  U_0$. As we also have that $\LBDp{s}(m_0,m_1,m_2)\in U_0$ and we
  know that $\exp_G$ is injective over $U_0$,
  point~\eqref{it:bdt_is_well_defined-relation_ldc} of the statement
  now follows from~\eqref{eq:bdt_is_well_defined-relation_ldc-exp_id}.
\end{proof}

\begin{example}\label{ex:DC_mu-logs}
  In the context of Example~\ref{ex:DC_mu-DC_form}, we have that
  \begin{equation*}
    \jgg := \lie{U(1)}= i\R \stext{ and } \exp_{U(1)}(i\zeta) =
    \exp(i\zeta).
  \end{equation*}
  If we take $U_0:=i(-\pi,\pi)$ and $V_e:=U(1)\SM\{-1\}$, it is an
  elementary fact that $\exp_{U(1)}|_{U_0}:U_0\rightarrow V_e$ is a
  diffeomorphism; the corresponding inverse map is $\Log(z)$, (the
  appropriate restriction of) the principal branch of the complex
  logarithm. The sets $U_0'$ and $V_e'$ constructed in the proof of
  Lemma~\ref{le:exp_is_diffeo_and_more} are
  $U_0'=i(-\frac{\pi}{3},\frac{\pi}{3})$ and
  $V_e'=\{z\in U(1) : \abs{\Arg(z)}\leq \frac{\pi}{3}\}$ (where
  $\Arg(z)$ is the argument of $z$ that lies in
  $(-\pi,\pi]$). Following Definition~\ref{def:adt}, the domain of
  $\LDClp{\mu}{s}$ is
  $\WDt'' = (\DCp{s})^{-1}(V_e') =
  \{(r_0,r_1)\in\R_{>0}^2:\abs{\jmod_{2\pi}((r_1-r_0)^\mu)}<\frac{\pi}{3}\}$,
  where $\jmod_{2\pi}(r)$ is the unique real number in $(-\pi,\pi]$
  congruent to $r$ modulo $2\pi$; then, for $(r_0,r_1)\in \WDt''$,
  \begin{equation*}
    \LDClp{\mu}{s}(r_0,r_1) = \Log(\DClp{\mu}{s}(r_0,r_1)) =
    \jmod_{2\pi}((r_1-r_0)^\mu). 
  \end{equation*}
  Similarly, following Definition~\ref{def:bdt}, we have
  $\ti{\WDt}'':= (\BDp{s})^{-1}(V_e) =
  \{(r_0,r_1,r_2)\in\R_{>0}^3:\jmod_{2\pi}(-(r_2-r_0)^\mu +
  (r_1-r_0)^\mu + (r_2-r_0)^\mu) \neq \pi\}$ and, for all
  $(r_0,r_1,r_2)\in \ti{\WDt}''$,
  \begin{equation*}
    \LBDlp{\mu}{s}(r_0,r_1,r_2) = \Log(\BDp{s}(r_0,r_1,r_2)) =
    \jmod_{2\pi}(-(r_2-r_0)^\mu + (r_1-r_0)^\mu + (r_2-r_0)^\mu).
  \end{equation*}
\end{example}


\subsection{Small singular cochains associated to $\LDCp{s}$ and
  $\LBDp{s}$}
\label{sec:small_singular_cochains_associated_to_LDCs_and_LBDs}

Now we reproduce the arguments of
Section~\ref{sec:discrete_connections_and_singular_cochains} to obtain
small singular cochains with values in $\jgg$ associated to $\LDCp{s}$
and $\LBDp{s}$.

Recall the open subset $\WDt''\subset \PBbase\times \PBbase$
introduced in Definition~\ref{def:adt}. In what follows, we work with
the $\WDt''$-small singular cochain complex
$(S_{n,\WDt''}(\PBbase,\jgg),\del_n^{\WDt''})$.

\begin{definition}
  With the notation as above, for any $1$-simplex
  $T_1\in S_{1,\WDt''}(\PBbase,\jgg)$ we define
  \begin{equation*}
    [\LDCp{s}](T_1) := \LDCp{s}(T_1(1,0),T_1(0,1)). 
  \end{equation*}
  Also, for any $2$-simplex $T_2\in S_{2,\WDt''}(\PBbase,\jgg)$ we define
  \begin{equation*}
    [\LBDp{s}](T_2) := \LBDp{s}(T_2(1,0,0),T_2(0,1,0),T_2(0,0,1)). 
  \end{equation*}
  As the $1$ and $2$ simplexes freely generate $S_{1,\WDt''}(\PBbase)$
  and $S_{2,\WDt''}(\PBbase)$ respectively, the previous formulas
  uniquely define $[\LDCp{s}]\in S^{1,\WDt''}(\PBbase,\jgg)$ and
  $[\LBDp{s}]\in S^{2,\WDt''}(\PBbase,\jgg)$.
\end{definition}

\begin{prop}
  With the same notation as above,
  $\delta_1^{\WDt''}[\LDCp{s}] = [\LBDp{s}]$ and
  $\delta_2^{\WDt''}[\LBDp{s}] = 0$.
\end{prop}

\begin{proof}
  As the $2$-simplexes generate $S_{2,\WDt''}(\PBbase)$, it suffices
  to verify that
  $(\delta_1^{\WDt''}[\LDCp{s}])(T_2) = [\LBDp{s}](T_2)$ for every
  such $2$-simplex $T_2$. After unraveling the definitions, this
  identity follows from
  point~\eqref{it:bdt_is_well_defined-relation_ldc} of
  Lemma~\ref{le:bdt_is_well_defined} and (a ``small version'' of) the
  formula in Example~\ref{ex:singular_12_cochains}. Thus, the first
  identity holds. The second one is a consequence of the first and of
  Lemma~\ref{le:properties_of_small_cochains}.
\end{proof}

\begin{prop}\label{prop:small_cochains_and_exp_G}
  The cochains $[\DCp{s}] \in S^{1,\VDt''}(X,G)$ and
  $[\BDp{s}] \in S^{2,\VDt''}(X,G)$ are naturally elements of
  $S^{1,\WDt''}(X,G)$ and $S^{2,\WDt''}(X,G)$ respectively. As such, we have
  \begin{equation*}
    [\DCp{s}] = (\exp_G)_* [\LDCp{s}] \stext{ and }
    [\BDp{s}] = (\exp_G)_* [\LBDp{s}],
  \end{equation*}
  where $(exp_G)_*$ is the homomorphism of cochain complexes defined
  in Section~\ref{sec:small_singular_chains_and_cochains} and induced
  by $\exp_G\in\hom(\jgg,G)$.
\end{prop}

\begin{proof}
  That the $\VDt''$-small cochains are also $\WDt''$-small cochains
  follows from $\WDt''\subset \VDt''$
  (Lemma~\ref{le:adt_is_well_defined}) and, then, because
  $S_{n,\WDt''}(\PBbase)\subset S_{n,\VDt''}(\PBbase)$. To check the
  identities, it suffices to see that they are satisfied on $1$ and
  $2$ simplexes $T_1\in S_{1,\WDt''}(\PBbase)$ and
  $T_2\in S_{2,\WDt''}(\PBbase)$. But then, on evaluation both
  identities are satisfied because of Definitions~\ref{def:adt}
  and~\ref{def:bdt}.
\end{proof}

\begin{theorem}\label{th:phases_and_loops-algebra}
  Let $\PBmap:\PBtotal\rightarrow \PBbase$ be a principal $G$-bundle
  with $G$ abelian and $\DC:\UDt\rightarrow G$ be a discrete
  connection on $\PBmap$. Given an open subset $V\subset \PBbase$ and
  $s:V\rightarrow \PBtotal$ a smooth section of $\PBmap$ as well as
  $\conj{m}\in V$, for any $m_\cdot\in \Omega_N(\conj{m})$ such that
  $(m_0,m_k,m_{k+1}) \in {\WDt''}^{(3)}$ for all $k=0,\ldots,N-1$, we
  have that the discrete holonomy phase around $m_\cdot$ is
  \begin{equation}\label{eq:phases_and_loops-algebra-DC}
    \HPd(m_\cdot) = \exp_G\left(-\int_{\ti{m}}[\LDCp{s}]\right),
  \end{equation}
  where $\ti{m} \in S_{1,\WDt''}(\PBbase)$ interpolates $m_\cdot$, that is,
  satisfies the conditions of Lemma~\ref{le:paths_and_simplexes}. If,
  in addition, there is $\ti{\sigma}\in S_{2,\WDt''}(\PBbase)$ so that
  $\del_2^{\WDt''}(\ti{\sigma}) = \ti{m}$, then
  \begin{equation}\label{eq:phases_and_loops-algebra-BD}
    \HPd(m_\cdot) = \exp_G\left(-\int_{\ti{\sigma}}[\LBDp{s}]\right).
  \end{equation}
\end{theorem}

\begin{proof}
  Observe that
  $\ti{m}\in S_{1,\WDt''}(\PBbase) \subset S_{1,\VDt''}(\PBbase)$ so
  that, by Theorem~\ref{th:phases_and_loops-group}, we
  have~\eqref{eq:phases_and_loops-group-DC}. Then, using
  Proposition~\ref{prop:small_cochains_and_exp_G} and the fact that
  $\exp_G$ is a homomorphism,
  \begin{equation*}
    \begin{split}
      \HPd(m_\cdot) =& \left(\int_{\ti{m}}[\DCp{s}]\right)^{-1} =
      ([\DCp{s}](\ti{m}))^{-1} = (((\exp_G)_*[\LDCp{s}])(\ti{m}))^{-1}
      = \exp_G([\LDCp{s}](\ti{m}))^{-1} \\=& \exp_G(-[\LDCp{s}](\ti{m}))
      = \exp_G\left(-\int_{\ti{m}}[\LDCp{s}]\right),
    \end{split}
  \end{equation*}
  proving~\eqref{eq:phases_and_loops-algebra-DC}. Then, if
  $\del_2^{\WDt''}(\ti{\sigma}) =
  \ti{m}$,~\eqref{eq:phases_and_loops-algebra-BD} follows immediately
  from~\eqref{eq:phases_and_loops-algebra-DC} and
  \begin{equation*}
    \int_{\ti{m}} [\LDCp{s}] =
    \int_{\del_2^{\WDt''}(\ti{\sigma})} [\LDCp{s}] =
    \int_{\ti{\sigma}} \delta_1^{\WDt''} [\LDCp{s}] =
    \int_{\ti{\sigma}} [\LBDp{s}].
  \end{equation*}
\end{proof}




\def\cprime{$'$} \def\polhk#1{\setbox0=\hbox{#1}{\ooalign{\hidewidth
  \lower1.5ex\hbox{`}\hidewidth\crcr\unhbox0}}} \def\cprime{$'$}
  \def\cprime{$'$}
\providecommand{\bysame}{\leavevmode\hbox to3em{\hrulefill}\thinspace}
\providecommand{\MR}{\relax\ifhmode\unskip\space\fi MR }
\providecommand{\MRhref}[2]{%
  \href{http://www.ams.org/mathscinet-getitem?mr=#1}{#2}
}
\providecommand{\href}[2]{#2}


\end{document}